\documentclass[10pt]{article}

\usepackage{comment,url,algorithm,algorithmic,graphicx,subcaption,relsize}
\usepackage{amssymb,amsfonts,amsmath,amsthm,amscd,dsfont,mathrsfs,mathtools,microtype,nicefrac,pifont}
\usepackage{float,psfrag,epsfig,color,url,hyperref}
\usepackage{upgreek}
\usepackage[dvipsnames]{xcolor}
\usepackage{epstopdf,bbm,mathtools,enumitem}
\usepackage[toc,page]{appendix}
\usepackage[mathscr]{euscript}
\usepackage[giveninits=true, maxnames=10, sorting=nyt, style=alphabetic]{biblatex}

\usepackage[top=1in, bottom=1in, left=1in, right=1in]{geometry}

\newcommand{\dee}{\mathop{\mathrm{d}\!}}

\def\balign#1\ealign{\begin{align}#1\end{align}}
\def\baligns#1\ealigns{\begin{align*}#1\end{align*}}
\def\balignat#1\ealign{\begin{alignat}#1\end{alignat}}
\def\balignats#1\ealigns{\begin{alignat*}#1\end{alignat*}}
\def\bitemize#1\eitemize{\begin{itemize}#1\end{itemize}}
\def\benumerate#1\eenumerate{\begin{enumerate}#1\end{enumerate}}

\newenvironment{talign*}
 {\csname align*\endcsname}
 {\endalign}
\newenvironment{talign}
 {\csname align\endcsname}
 {\endalign}

\def\balignst#1\ealignst{\begin{talign*}#1\end{talign*}}
\def\balignt#1\ealignt{\begin{talign}#1\end{talign}}
\let\originalleft\left
\let\originalright\right
\renewcommand{\left}{\mathopen{}\mathclose\bgroup\originalleft}
\renewcommand{\right}{\aftergroup\egroup\originalright}

\def\tinycitep*#1{{\tiny\citep*{#1}}}
\def\tinycitealt*#1{{\tiny\citealt*{#1}}}
\def\tinycite*#1{{\tiny\cite*{#1}}}
\def\smallcitep*#1{{\scriptsize\citep*{#1}}}
\def\smallcitealt*#1{{\scriptsize\citealt*{#1}}}
\def\smallcite*#1{{\scriptsize\cite*{#1}}}

\def\<{\left\langle} % Angle brackets
\def\>{\right\rangle}

\def\polylog{\operatorname{polylog}}
\DeclareSymbolFont{rsfs}{U}{rsfs}{m}{n}
\DeclareSymbolFontAlphabet{\mathscrsfs}{rsfs}
\ifdefined\nonewproofenvironments\else
\ifdefined\ispres\else
\newtheorem{theorem}{Theorem}
\newtheorem{lemma}[theorem]{Lemma}
\newtheorem{corollary}[theorem]{Corollary}

\renewenvironment{proof}{\noindent\textbf{Proof.}\hspace*{.3em}}{\qed \vspace{.1in}}
\newenvironment{proof-sketch}{\noindent\textbf{Proof Sketch}
  \hspace*{1em}}{\qed\bigskip\\}
\newenvironment{proof-idea}{\noindent\textbf{Proof Idea}
  \hspace*{1em}}{\qed\bigskip\\}
\newenvironment{proof-of-lemma}[1][{}]{\noindent\textbf{Proof of Lemma {#1}}
  \hspace*{1em}}{\qed\\}
\newenvironment{proof-of-theorem}[1][{}]{\noindent\textbf{Proof of Theorem {#1}}
  \hspace*{1em}}{\qed\\}
\newenvironment{proof-attempt}{\noindent\textbf{Proof Attempt}
  \hspace*{1em}}{\qed\bigskip\\}

\newenvironment{remark}{\noindent\textbf{Remark.}
  \hspace*{0em}}{\smallskip}%\bigskip}

\fi

\newtheorem{proposition}[theorem]{Proposition}

\theoremstyle{definition}
\newtheorem{example}[theorem]{Example}
\fi
\makeatletter
\@addtoreset{equation}{section}
\makeatother

\hypersetup{
  colorlinks,
  linkcolor={red!50!black},
  citecolor={blue!50!black},
  urlcolor={blue!80!black}
}

\newcommand{\cmark}{\ding{51}}
\newcommand{\xmark}{\ding{55}}

\renewcommand{\Pr}[1]{\mathbb{P}\left( #1 \right)}

\newcommand{\ind}[1]{{\mathbbm{1}}_{\{ #1 \}} }
\usepackage{custom}
\usepackage[normalem]{ulem}

\allowdisplaybreaks{}
\makeatletter
\renewcommand{\paragraph}{%
  \@startsection{paragraph}{4}%
  {\z@}{1.25ex \@plus 1ex \@minus .2ex}{-1em}%
  {\normalfont\normalsize\bfseries}%
}
\makeatother

\mathtoolsset{showonlyrefs}

\begin{document}

\title{Analysis of Langevin Monte Carlo  from Poincar\'e to Log-Sobolev}

 \author{\!\!\!\!\!
 Sinho Chewi\thanks{
  School of Mathematics at
  Institute for Advanced Study, \texttt{schewi@ias.edu}
 }
 \ \ \
 Murat A. Erdogdu\thanks{
  Department of Computer Science at
  University of Toronto, and Vector Institute, \texttt{erdogdu@cs.toronto.edu}
 }
 \ \ \
 Mufan (Bill) Li\thanks{
  Department of Statistical Sciences at
  University of Toronto, and Vector Institute, \texttt{mufan.li@mail.utoronto.ca}
}
 \ \ \
Ruoqi Shen\thanks{
  Paul G.\ Allen School of Computer Science and Engineering at University of Washington, \texttt{shenr3@cs.washington.edu}
}
 \ \ \ 
Matthew S.\ Zhang\thanks{
  Department of Computer Science at
  University of Toronto, and Vector Institute, \texttt{matthew.zhang@mail.utoronto.ca}
}
}

\maketitle

\begin{abstract}
Classically, the continuous-time Langevin diffusion converges exponentially fast to its stationary distribution $\pi$ under the sole assumption that $\pi$ satisfies a Poincar\'e inequality. Using this fact to provide guarantees for the discrete-time Langevin Monte Carlo (LMC) algorithm, however, is considerably more challenging due to the need for working with chi-squared or R\'enyi divergences, and prior works have largely focused on strongly log-concave targets. In this work, we provide the first convergence guarantees for LMC assuming that $\pi$ satisfies either a Lata\l{}a--Oleszkiewicz or modified log-Sobolev inequality, which interpolates between the Poincar\'e and log-Sobolev settings. Unlike prior works, our results allow for weak smoothness and do not require convexity or dissipativity conditions.
\end{abstract}

\section{Introduction}

The task of sampling from a target distribution $\pi \propto \exp(-V)$ on $\R^d$, known only up to a normalizing constant, is fundamental in many areas of scientific computing~\cite{mackay2003information,robertcasella2004mcmc,brooks2011handbook,gelman2013bayesian}. As such, there has been a considerable amount of research dedicated to this task, yielding precise and non-asymptotic algorithmic guarantees when the potential $V$ is strongly convex~\cite[see, e.g.,][]{dalalyan2017analogy, durmusmajewski2019lmcconvex, dwivedi2019log, shenlee2019randomizedmidpoint, hebalasubramanianerdogdu2020rm, leeshentian2020logsmooth, chewietal2021mala, caoluwang2021uldlowerbd, Chewietal22Sampling1D}. Many distributions encountered in practice, however, are non-log-concave, and it is therefore of central importance to provide sampling guarantees for such distributions. In this work, we address this problem by working under the assumption that $\pi$ satisfies a suitable functional inequality, which we now motivate.

The canonical sampling algorithm, Langevin Monte Carlo (LMC), is based on a discretization of the continuous-time Langevin diffusion, which is the solution to the stochastic differential equation
\begin{align}\label{eq:langevin}
    \D z_t
    &= -\nabla V(z_t) \, \D t + \sqrt 2 \, \D B_t\,.
\end{align}
Here, ${(B_t)}_{t\ge 0}$ is a standard Brownian motion in $\R^d$.
Classically, if $\pi$ satisfies a functional inequality such as a Poincar\'e inequality or a log-Sobolev inequality, then the law of the Langevin diffusion~\eqref{eq:langevin} converges exponentially fast to the target distribution $\pi$~\cite{bakrygentilledoux2014}. Namely, a Poincar\'e inequality implies exponential convergence in chi-squared divergence, whereas a log-Sobolev inequality (which is stronger than a Poincar\'e inequality) implies exponential convergence in KL divergence.

The class of measures satisfying a Poincar\'e inequality is quite large, including all strongly log-concave measures (due the Bakry--\'{E}mery criterion) and, more generally, all log-concave measures~\cite{kls1995, bobkov1999logconcave, chen2021kls}.
It also includes many examples of non-log-concave distributions such as Gaussian convolutions of measures with bounded support~\cite{bardetetal2018gaussianconvolutions, chenchewinilesweed2021dimfreelsi}, and it is closed under bounded perturbations of the log-density. Owing to its broad applicability and its favorable continuous-time convergence properties, this class of measures is thus a natural goal for providing quantitative guarantees for non-log-concave sampling.

\paragraph*{Sampling guarantees under functional inequalities.} 
Convergence of Markov chains under a functional inequality, such as a log-Sobolev inequality, has a storied literature; see, e.g.,~\cite{miclo1992recuit, miclo1997remarques, miclo1999majoration}. Our work, however, is primarily inspired by~\cite{vempalawibisono2019ula}, which advocated the use of a functional inequality paired with a smoothness condition as a minimal set of assumptions for obtaining sampling guarantees; in their work, Vempala and Wibisono prove convergence of LMC in KL divergence under a log-Sobolev inequality. This result was then improved using the proximal Langevin algorithm under higher-order smoothness in~\cite{wibisono2019proximal} and extended to  Riemannian manifolds in~\cite{LiErd23RiemLangevin}.

Despite the appeal of this program, however, the majority of works on non-log-concave sampling instead make an additional assumption on the growth of the potential known as a dissipativity condition~\cite[e.g.,][]{raginskyrakhlintelgarsky2017sgld,erdogdu2018global, erdogduhosseinzadeh2021tailgrowth, erdogduhosseinzadehzhang2022chisquare, mouetal2022langevin, NguDanChe23LMCWeaklySmooth}. A representative example of such a condition is $\langle \nabla V(x), x \rangle \ge a \, \|x\|^2 - b$ for some constants $a, b > 0$.
Although useful for discretization proofs, dissipativity conditions are arguably less natural from the standpoint of the quantitative theory of Markov processes~\cite{bakrygentilledoux2014}, and ultimately redundant in the presence of an appropriate functional inequality. Other drawbacks include the fact that $b$ is typically dimension-dependent, and that dissipativity conditions are not as stable under perturbations (see Section~\ref{scn:examples} for an example). Hence, we avoid such conditions in our work.

In our first main result (Theorem~\ref{thm:main}), we assume that the target $\pi$ satisfies a Lata\l{}a--Oleszkiewicz inequality (LOI) with parameter $\alpha \in [1,2]$. LO inequalities are well-studied functional inequalities that elegantly interpolate between Poincar\'e and log-Sobolev inequalities~\cite{latalaoleszkiewicz2000loineq}. Notably, the $\alpha = 1$ case reduces to the Poincar\'e inequality, while the $\alpha = 2$ case reduces to the log-Sobolev inequality; intermediate values of $\alpha$ enable capturing potentials with growth $V(x) \approx \norm x^\alpha$ (see Section~\ref{scn:lo-cnts}). We also complement our result by proving a sampling guarantee (Theorem~\ref{thm:main_mlsi}) under the modified log-Sobolev inequalities considered in~\cite{erdogduhosseinzadeh2021tailgrowth}, which is useful for treating examples in which the LOI constant is dimension-dependent.

\paragraph*{Towards weaker notions of smoothness.} Since the assumption of a Poincar\'e inequality allows for a variety of non-convex potentials with at least linear growth, it is restrictive to pair this assumption with the gradient Lipschitz assumption which is usually invoked in the sampling literature. Hence, following~\cite{devolder2014first,nesterov2015universal,chatterjietal2020lmcwithoutsmoothness, erdogduhosseinzadeh2021tailgrowth}, we instead assume that $\nabla V$ is H\"older-continuous with exponent $s \in (0, 1]$.

\paragraph*{An analysis in R\'enyi divergence.} We now describe the main technical challenge of this work. Recall that a log-Sobolev inequality (LSI) implies exponential ergodicity of the diffusion~\eqref{eq:langevin} in KL divergence, and consequently the analysis of LMC under a LSI naturally proceeds with the KL divergence as the performance metric~\cite{vempalawibisono2019ula, wibisono2019proximal, LiErd23RiemLangevin}. Similarly, a Poincar\'e inequality implies exponential ergodicity of~\eqref{eq:langevin} in chi-squared divergence, and accordingly we analyze LMC in chi-squared divergence, or equivalently, in R\'enyi divergence. In turn, the techniques we develop for the analysis may be useful for other situations in which only a Poincar\'e-type inequality is available, such as the state-of-the-art convergence rate for the underdamped Langevin diffusion~\cite{CaoLuWan23Underdamped} or for the mirror-Langevin diffusion~\cite{chewietal2020mirrorlangevin}.

Via standard comparison inequalities, a convergence guarantee in R\'enyi divergence implies convergence for other common divergences (e.g., total variation distance, $2$-Wasserstein distance, or KL divergence), and is therefore more desirable. Of particular interest in this regard is the role of R\'enyi divergence guarantees for providing ``warm starts'' for high-accuracy samplers such as the Metropolis-adjusted Langevin algorithm~\cite{chewietal2021mala, wuschche22mala} and the zigzag sampler~\cite{luwang2022zigzag} (see Section~\ref{sec:warm-start}).

Unfortunately, working with R\'enyi divergences introduces substantial new technical hurdles as it prevents the use of standard coupling-based discretization arguments; as such, there are not many prior works to draw upon. The convergence of the diffusion~\eqref{eq:langevin} in R\'enyi divergence was first proven in~\cite{caolulu2019renyi, vempalawibisono2019ula}, although the analytical techniques have had a long history, dating back to the earliest works on hypercontractivity \cite{gross1975logarithmic, arnold2001convex}. 
The paper~\cite{vempalawibisono2019ula} also takes a first step towards discretization by introducing a technique based on differential inequalities for the R\'enyi divergence for a continuous-time interpolation of LMC\@. 
Although this strategy succeeds for obtaining KL convergence under an LSI, it falls short for R\'enyi divergence. Indeed, the analysis of~\cite{vempalawibisono2019ula} only holds under the (currently unverifiable) assumption that the \emph{biased stationary distribution} of the LMC algorithm satisfies a Poincar\'e inequality. 
Moreover, their result only establishes quantitative convergence of LMC to its biased limit; to recover a convergence guarantee to $\pi$, this also requires an estimate of the ``R\'enyi bias'' (the R\'enyi divergence between the biased stationary distribution and $\pi$), which was unresolved. 
Instead,~\cite{ganeshtalwar2020renyi} provided the first R\'enyi guarantee for LMC by using the adaptive composition theorem 
to control the discretization error, albeit suboptimally. Subsequently, their result was sharpened in~\cite{erdogduhosseinzadehzhang2022chisquare} via a two-stage analysis combining the tools of~\cite{vempalawibisono2019ula} and \cite{ganeshtalwar2020renyi}.

In this paper, we first provide a genuine R\'enyi convergence guarantee for LMC under an LSI, thereby yielding a stronger result than~\cite{vempalawibisono2019ula,ganeshtalwar2020renyi, erdogduhosseinzadehzhang2022chisquare} with a shorter and more elegant proof. 
We further extend this to the case when $\pi$ is log-concave, but this technique is unable to cover the setting of a weaker functional inequality and smoothness condition. For this, we instead draw inspiration from the stochastic calculus-based analysis of~\cite{dalalyantsybakov2012sparseregression, chewietal2021mala}. At the heart of our proofs is the introduction of new change-of-measure inequalities which intriguingly rely on the very fact that the analysis is carried out in R\'enyi divergence (and not any weaker metric). Thus, although the use of R\'enyi divergences introduces new technical obstructions, it also provides the key tool for overcoming them.

\subsection{Contributions}

\paragraph*{Convergence of the diffusion under functional inequalities.}
Our first contribution is to establish quantitative R\'enyi convergence bounds for the Langevin diffusion~\eqref{eq:langevin} under the following functional inequalities: (1) the Lata\l{}a--Oleszkiewicz inequalities (LOI)~\cite{latalaoleszkiewicz2000loineq},
which
interpolate between the Poincar\'e and log-Sobolev inequalities (Theorem~\ref{thm:lo_renyi}), and (2) the modified log-Sobolev inequality (MLSI)~\cite{erdogduhosseinzadeh2021tailgrowth}  (Theorem~\ref{thm:mlsi_cont_time}).
LOI and MLSI have relative merits, and they capture the tail behavior of the potential, providing an accurate characterization of the speed of convergence for both the diffusion and its discretization.

\paragraph*{Improved guarantees for LMC under an LSI or log-concavity.}
As our second principal contribution (Theorem~\ref{thm:lsi_result}), we provide an elegant proof that under an LSI, the LMC algorithm (with appropriate step size) achieves $\varepsilon$ error in $q$-R\'enyi divergence in $\widetilde O(C_{\msf{LSI}}^2 L^2 dq/\varepsilon)$ iterations, where $C_{\msf{LSI}}$ is the LSI constant and $L$ is the Lipschitz constant of $\nabla V$. This improves upon past works in several respects. First, in the LSI case, without any additional assumptions such as dissipativity or strong convexity, a R\'enyi convergence guarantee for LMC was previously unknown; thus, our work strengthens~\cite{vempalawibisono2019ula} by proving convergence in a stronger metric (R\'enyi divergence rather than KL divergence). Second, even when the target $\pi$ is strongly log-concave, our proof is sharper in $L$, $q$, and the strong convexity constant, as well as being significantly shorter than the prior works~\cite{ganeshtalwar2020renyi, erdogduhosseinzadehzhang2022chisquare}; moreover, our guarantee for fixed step size LMC does not degrade if the number of iterations is taken too large. These improvements follow from an entirely novel argument which utilizes a change-of-measure interpretation of the R\'enyi divergence (see Theorem \ref{thm:lsi_result}). As a corollary, we resolve an open question of~\cite{vempalawibisono2019ula} on the size of the ``R\'enyi bias'' in this setting (see Corollary~\ref{cor:renyi_bias}).

With additional effort, we are able to extend the techniques to the case when $\pi$ is (weakly) log-concave, and we obtain a guarantee with explicit dependence on the Poincar\'e constant of $\pi$; see Theorem~\ref{thm:log_concave_result}. Our result is the state-of-the-art guarantee for LMC for sampling from isotropic log-concave targets.

\paragraph*{Convergence of LMC under a functional inequality and weak smoothness.}
Our main contribution is to provide sampling guarantees assuming that the potential has a H\"older-continuous gradient of exponent $s \in (0, 1]$ and that $\pi$ either satisfies LOI (Theorem~\ref{thm:main}) or MLSI (Theorem~\ref{thm:main_mlsi}). As noted previously, these assumptions are considerably more general than what are usually considered in the sampling literature and do not require dissipativity. In particular, Theorem~\ref{thm:main} completes the program of~\cite{vempalawibisono2019ula} by establishing the first sampling guarantees for LMC under a Poincar\'e inequality and a weak smoothness condition.

Generically, our final rate is $\widetilde O(d^{(2/\alpha) \, (1+1/s) - 1/s}/\varepsilon^{1/s})$, where $s$ is the H\"older continuity exponent of $\nabla V$ and $\alpha$ captures the growth of the potential at infinity.
We give several illustrative examples in Section~\ref{scn:examples}.

\paragraph*{Accuracy and bias of the Euler--Maruyama scheme.}
Our result in Theorem~\ref{thm:main_disc_bd} provides an explicit bound on the error of the Euler--Maruyama scheme, a well-studied problem within numerical stochastic analysis. 
Classical results in this domain mainly focus on the weak/strong
  error as a function of the step size~\cite{kloeden1992stochastic, milstein1994numerical, higham2002strong, mattingly2002ergodicity,milstein2004stochastic}, without
  quantifying the dependence on important problem parameters such as the dimension.
  On the other hand, 
  recent works with explicit dimension bounds were either in weaker metrics or relied on additional assumptions such as dissipativity or strong convexity \cite{mattingly2002ergodicity,raginskyrakhlintelgarsky2017sgld, cheng2018convergence, erdogduhosseinzadehzhang2022chisquare}.
 
\paragraph*{Summary.}
We note that although we present rate estimates under numerous regimes for the target $\pi$, our results can be interpreted as lying on a smooth continuum, seen through the following inclusions.

\begin{align*}
    \underset{\text{Theorem 6}}{\text{log-concave}} \subset \underset{\text{Theorem 7}}{\text{Poincar\'e inequality}} \supset \underset{\text{Theorem 7}}{\text{LO inequality}} \supset \underset{\text{Theorem 4}}{\text{log-Sobolev inequality}}\,. 
\end{align*}
We present these as separate theorems since they each require different analytical techniques. However, our results can clearly be seen as a ``master schema'' for obtaining non-asymptotic rate estimates under a very broad range of assumptions.

\subsection{Notation and organization}

Throughout the paper, $\pi\propto\exp(-V)$ denotes the target distribution on $\R^d$; the function $V : \R^d\to\R$ is referred to as the ``potential''. We abuse notation by identifying a measure with its density (w.r.t.\ Lebesgue measure on $\R^d$). We write $a \lesssim b$ and $a = O(b)$ to indicate that $a \le Cb$ for a universal constant $C > 0$; also, we use $\widetilde O(\cdot)$ as a shorthand for $O(\cdot) \log^{O(1)}(\cdot)$. 
Following the standard notation, expressions like $O(a), \widetilde{O}(a)$ will denote terms like $Ca, Ca \log^{O(1)}(a)$ for some absolute constant $C > 0$.
A polylogarithmic dependency is any dependency of order $ \log^{O(1)}(\cdot)$, and we denote it by $\polylog(\cdot)$. Likewise, $a \gtrsim b$ or $a = \Omega(b)$ is equivalent to $a \ge Cb$ for a universal constant $C > 0$, and $\widetilde{\Omega}$ is defined analogously to $\widetilde O$. Finally, we say $a \asymp b$ or $a = \Theta(b)$ if simultaneously $a \lesssim b$ and $a \gtrsim b$ holds, with an analogous definition for $\widetilde\Theta$.

We define KL and $\chi^2$ divergences between two measures $\mu$ and $\nu$, respectively as
\[
    \msf{KL}(\mu \mmid \nu)  \deq  \int \log \frac{\D \mu}{\D \nu}\,\D \mu\,\quad \text{ and }\quad\chi^2(\mu \mmid \nu)  \deq  \int \Bigl(\frac{\D \mu}{\D \nu} - 1 \Bigr)^2\, \D \nu \,, 
\]
if $\mu \ll \nu$, and $\msf{KL}(\mu \mmid \nu)=\chi^2(\mu \mmid \nu) \deq +\infty$ otherwise.
We also define the squared $2$-Wasserstein distance and the total variation distance between $\mu$ and $\nu$, respectively as
\[
    W_2^2(\mu, \nu)  \deq  \inf_{\gamma \in \Gamma(\mu, \nu)} \E_{(x,y) \sim \gamma}\bigl[\norm{x-y}^2 \bigr]\,
    \quad\text{ and }\quad \norm{\mu - \nu}_{\msf{TV}}  \deq  \sup_{A\in\mathcal{B}(\R^d)} \abs{\mu(A) - \nu(A)}\,,
\]
where $\Gamma(\mu, \nu)$ is the set of all probability distributions on $\R^d \times \R^d$ with marginals $\mu, \nu$ with respect to the first and second coordinates respectively, and the supremum is taken over all Borel measurable sets of $\R^d$.

The paper is organized as follows. In Section~\ref{scn:continuous-time}, we begin by reviewing functional inequalities and their implications for the continuous-time convergence of the diffusion~\eqref{eq:langevin} in R\'enyi divergence. We then state our main theorems on the LMC algorithm in Section~\ref{scn:results}. Section~\ref{scn:applications} provides several illustrative examples and useful application areas for our theorems, and Section~\ref{scn:overview} provides a high-level overview for the proofs. We fill in the proof details in Section~\ref{scn:proofs}, with additional lemmas deferred to the Appendix. We conclude in Section~\ref{scn:conclusion} with a discussion of future directions of research.

\section{Functional inequalities and continuous-time convergence}\label{scn:continuous-time}

Our focus in this section is the convergence of the continuous-time Langevin diffusion \eqref{eq:langevin} under various functional inequalities.
Throughout the paper, we use
the R\'enyi divergence as a measure of distance between two probability laws. 
The R\'enyi divergence of order $q \in (1,\infty)$ of $\mu$ from $\pi$ is defined to be
\begin{align*}
    \eu R_q(\mu \mmid \pi)
    &\deq \frac{1}{q-1} \ln {\Bigl\lVert \frac{\D \mu}{\D \pi} \Bigr\rVert_{L^q(\pi)}^q}\,,
\end{align*}
where $\eu R_q(\mu \mmid \pi)$ is understood to be $+\infty$ if $\mu\not\ll \pi$. Although the R\'enyi divergence is not a genuine metric (it is not symmetric and it does not satisfy the triangle inequality), it has the property that $\eu R_q(\mu \mmid \pi) \ge 0$, with equality if and only if $\mu = \pi$. 

R\'enyi divergence is monotonic in the order, that is, if $1 < q \le q'$, then we have $\eu R_q \le \eu R_{q'}$.
Notable special cases include:
\begin{itemize}
    \item as $q\searrow 1$, the R\'enyi divergence $\eu R_q(\mu\mmid \pi)$ approaches the KL divergence $\msf{KL}(\mu \mmid \pi)$;
    \item for $q=2$, the R\'enyi divergence is closely related to the chi-squared divergence through $\eu R_2(\mu \mmid \pi) = \ln(1+\chi^2(\mu \mmid \pi))$;
    \item and as $q\to\infty$, we have $\eu R_q(\mu \mmid \pi) \to \eu R_\infty(\mu \mmid \pi) \deq \ln{\norm{\frac{\D \mu}{\D \pi}}_{L^\infty(\pi)}}$.
\end{itemize}

These divergences are particularly of interest because they conveniently upper bound a variety of distance measures to be discussed shortly.

\subsection{Poincar\'e and log-Sobolev inequalities}

In the context of sampling, the most well-studied functional inequalities are the \emph{Poincar\'e inequality} (PI) and the \emph{log-Sobolev inequality} (LSI). We say that $\pi$ satisfies a PI with constant $C_{\msf{PI}}$ if, for all smooth functions $f : \R^d\to\R$, it holds that
\begin{align}\label{eq:pi}\tag{$\msf{PI}$}
    \var_\pi(f)
    &\le C_{\msf{PI}} \E_\pi[\norm{\nabla f}^2]\,.
\end{align}
Similarly, we say that $\pi$ satisfies an LSI with constant $C_{\msf{LSI}}$ if for all smooth $f : \R^d\to\R$,
\begin{align}\label{eq:lsi}\tag{$\msf{LSI}$}
    \ent_\pi(f^2)
    &\le 2C_{\msf{LSI}} \E_\pi[\norm{\nabla f}^2]\,,
\end{align}
where $\ent_\pi(f^2) \deq \E_\pi[f^2 \ln(f^2/\E_\pi(f^2))]$. An LSI implies a PI with the same constant.

These functional inequalities are classically related to the ergodicity properties of the Langevin diffusion~\eqref{eq:langevin}. Indeed, if $\pi_t$ denotes the law of the diffusion at time $t$, then a PI is equivalent to
\begin{align*}
    \chi^2(\pi_t \mmid \pi)
    &\le \exp\Bigl(-\frac{2t}{C_{\msf{PI}}}\Bigr) \, \chi^2(\pi_0 \mmid \pi)\,, \qquad\text{for all}~t\ge 0\,,
\end{align*}
whereas an LSI is equivalent to
\begin{align*}
    \msf{KL}(\pi_t \mmid \pi)
    &\le \exp\Bigl(-\frac{2t}{C_{\msf{LSI}}}\Bigr) \, \msf{KL}(\pi_0 \mmid \pi)\,, \qquad\text{for all}~t\ge 0\,.
\end{align*}

Functional inequalities are particularly useful for high-dimensional non-log-concave sampling because they tensorize (if two measures satisfy the same functional inequality, then their product also satisfies the functional inequality with the same constant) and they are stable under common operations such as bounded perturbation (replacing the potential $V$ with $\widetilde V$, with $\sup{\abs{V - \widetilde V}} < \infty$) and Lipschitz mapping (replacing $\pi$ with $T_\# \pi$ where $T : \R^d\to\R^d$ is Lipschitz and the pushforward $T_\# \pi$ is the distribution of $T(x)$ when $x\sim\pi$). We refer to~\cite{bakrygentilledoux2014} for a comprehensive treatment.

Before stating the convergence results, we remark that under~\eqref{eq:pi}, the result of~\cite{liu2020poincare} together with standard comparison inequalities imply
\begin{align*}
    \max\Bigl\{ 2 \, \norm{\mu-\pi}_{\msf{TV}}^2, \; \ln\bigl(1+\frac{1}{2C_{\msf{PI}}} \, W_2^2(\mu,\pi)\bigr), \; \msf{KL}(\mu \mmid \pi)\Bigr\}
    &\le \eu R_2(\mu \mmid \pi)\,.
\end{align*}
Note that in the Poincar\'e case, a $\msf T_2$ transportation inequality does not necessarily hold, so a KL guarantee does not imply a matching $W_2$ guarantee; by working with R\'enyi divergences, we are able to provide a unified guarantee for all of these metrics simultaneously.

Improving upon the prior result of~\cite{caolulu2019renyi},~\cite{vempalawibisono2019ula} showed that these inequalities also imply R\'enyi convergence for the diffusion.

\begin{theorem}[{\cite[Theorems 3 and 5]{vempalawibisono2019ula}}]\label{thm:pi_lsi_renyi}
    Let $q \ge 2$, and let $\pi_t$ denote the law of the continuous-time Langevin diffusion~\eqref{eq:langevin} at time $t$.
    \begin{enumerate}
        \item If $\pi$ satisfies~\eqref{eq:lsi}, then
        \begin{align*}
            \partial_t \eu R_q(\pi_t \mmid \pi)
            &\le -\frac{2}{qC_{\msf{LSI}}} \, \eu R_q(\pi_t \mmid \pi)\,.
        \end{align*}
        \item If $\pi$ satisfies~\eqref{eq:pi}, then
        \begin{align*}
            \partial_t \eu R_q(\pi_t \mmid \pi)
            &\le - \frac{2}{qC_{\msf{PI}}} \times \begin{cases} 1\,, & \text{if}~\eu R_q(\pi_t \mmid \pi) \ge 1\,, \\ \eu R_q(\pi_t \mmid \pi)\,, & \text{if}~\eu R_q(\pi_t \mmid \pi) \le 1\,. \end{cases}
        \end{align*}
    \end{enumerate}
\end{theorem}
The above result states that under LSI, 
the R\'enyi divergence decays exponentially fast whereas
under PI, dissipation can be explained in two phases; an initial phase of \emph{slow} decay
followed by exponential convergence.
Thus, to obtain $\eu R_q(\pi_T \mmid \pi) \le \varepsilon$, it suffices to have
\begin{align*}
        \emph{\text{1.}}\ \  T \ge  \Omega\Bigl( q C_{\msf{LSI}}  \ln\frac{{\eu R_q(\pi_0 \mmid \pi)}}{\varepsilon}\Bigr)\,\ \ \qquad\text{ and }\qquad\ \ \
        \emph{\text{2.}}\ \  T \ge \Omega\Bigl( q C_{\msf{PI}} \, \Bigl({\eu R_q(\pi_0 \mmid \pi) +\ln\frac{1}{\varepsilon}}\Bigr) \Bigr)
\end{align*}
under LSI and PI respectively.
Here and below, the statement ``it suffices to have $T \ge \Omega(T_0)$'' means that there exists an unspecified universal constant $C > 0$ such that if $T \ge CT_0$, then the desired conclusion (here, R\'enyi divergence at most $\varepsilon$) holds.

\subsection{Lata\l{}a--Oleszkiewicz inequalities}\label{scn:lo-cnts}

In this paper, in order to interpolate between the Poincar\'e and log-Sobolev cases, we consider a family of functional inequalities known as Lata\l{}a--Oleszkiewicz inequalities (LOI)~\cite{latalaoleszkiewicz2000loineq}. We say that $\pi$ satisfies an LOI of order $\alpha \in [1,2]$ and constant $\CLOI$ if for all smooth $f : \R^d\to\R$,
\begin{align}\label{eq:lo}\tag{$\msf{LOI}$}
    \sup_{p \in (1,2)} \frac{\E_\pi(f^2) - {\E_\pi(f^p)}^{2/p}}{{(2-p)}^{2 \, (1-1/\alpha)}} \le \CLOI \E_\pi[\norm{\nabla f}^2]\,.
\end{align}
An LOI of order $1$ is equivalent to a PI, and an LOI of order $2$ is equivalent to an LSI. More generally, an LOI of order $\alpha$ captures measures whose potentials ``have tail growth $\alpha$''; indeed, two notable examples of distributions satisfying the LOI of order $\alpha$ are $\pi(x) \propto \exp(-\sum_{i=1}^d \abs{x_i}^\alpha)$ and $\pi(x) \propto \exp(-\norm x^\alpha)$~\cite{latalaoleszkiewicz2000loineq, barthe2001levelsofconcentration}. LO inequalities are well-studied because they capture intermediate forms of concentration and are related to a number of other important inequalities, such as Sobolev inequalities; we refer readers to~\cite{latalaoleszkiewicz2000loineq, barthe2001levelsofconcentration, bartheroberto2003sobolevrealline, chafai2004phientropies, boucheronetal2005momentineq, wang2005generalizationpoincare, barthecattiauxroberto2006interpolated, barthecattiauxroberto2007isoperimetrybetween, cattiauxgentilguillin2007weaklsi, gozlan2010poincare}. 

As our first result in this section, we extend Theorem~\ref{thm:pi_lsi_renyi} to cover LOI\@. Our proof, which uses as an intermediary the super Poincar\'e inequality introduced in~\cite{wang2000superpoincare}, is deferred to Section~\ref{scn:lo_renyi_pf}.

\begin{theorem}\label{thm:lo_renyi}
    Let $q \ge 2$, and let $\pi_t$ denote the law of the continuous-time Langevin diffusion~\eqref{eq:langevin} at time $t$. Suppose that $\pi$ satisfies~\eqref{eq:lo} with order $\alpha$.
    Then,
    \begin{align*}
        \partial_t \eu R_q(\pi_t \mmid \pi)
        &\le - \frac{1}{68q \CLOI} \times \begin{cases} {\eu R_q(\pi_t \mmid \pi)}^{2-2/\alpha}\,, & \text{if}~\eu R_q(\pi_t \mmid \pi) \ge 1\,, \\ \eu R_q(\pi_t \mmid \pi)\,, & \text{if}~\eu R_q(\pi_t \mmid \pi) \le 1\,. \end{cases}
    \end{align*}
\end{theorem}

\noindent The above theorem 
can be used to obtain $\eu R_q(\pi_T \mmid \pi) \le \varepsilon$ whenever
\begin{align*}
        T \ge \Omega\Bigl( q \CLOI \, \Bigl( \frac{{\eu R_q(\pi_0 \mmid \pi)}^{2/\alpha - 1} - 1}{2/\alpha - 1} + \ln \frac{1}{\varepsilon}\Bigr)\Bigr)\,;
\end{align*}
we refer to Lemma~\ref{lem:cont_time_convergence} for details.
We also remark that Theorem~\ref{thm:lo_renyi} reduces to
Theorem~\ref{thm:pi_lsi_renyi} in the edge cases $\alpha=2$ (LSI) and $\alpha=1$ (PI) up to an absolute constant.
For $\alpha\in(1,2)$, the initial phase of convergence interpolates between the \emph{slow} decay induced by PI and the exponential decay under LSI.

\subsection{Modified log-Sobolev inequalities}

In addition, we also consider the modified log-Sobolev inequality (MLSI) of~\cite{erdogduhosseinzadeh2021tailgrowth}. The MLSI of order $\alpha_0 \in [-1, 2]$ states that for all $f : \R^d\to\R$ with $\E_\pi(f^2) = 1$,
\begin{align}\label{eq:mlsi}\tag{$\msf{MLSI}$}
    \ent_\pi(f^2)
    &\le 2C_{\msf{MLSI}}\; \inf_{p\ge 2} \bigl\{{\E_\pi[\norm{\nabla f}^2]}^{1-\delta(p)} \, {\widetilde{\mf m}_p\bigl((1+f^2)\pi\bigr)}^{\delta(p)}\bigr\}\,,
\end{align}
where $\delta(p)$ and $\widetilde{\mf m}_p(\mu)$ for a measure $\mu$ (not necessarily a probability measure) are given as
\begin{align*}
\delta(p) \deq \frac{2-\alpha_0}{p+2-2\alpha_0}, \quad\quad
    \widetilde{\mf m}_p(\mu) \deq \int {(1+\norm \cdot^2)}^{p/2} \, \D \mu\,.
\end{align*}
{The inequality \eqref{eq:mlsi} is a careful refinement of \cite{toscani2000trend},
and provides convergence guarantees for both the Langevin diffusion
and LMC under various tail growth conditions~\cite{erdogduhosseinzadeh2021tailgrowth}.
It is similar to log-Nash inequalities
\cite{bertini1999coercive,zegarlinski2001entropy}, 
yet the main focus of the latter is infinite-dimensional
semigroups. 
We consider \eqref{eq:mlsi} as used in~\cite{erdogduhosseinzadeh2021tailgrowth}
since other MLSI-type
results are stated by absorbing various dimension-dependent
constants into $C_{\msf{MLSI}}$, and thus they cannot
provide sharp rates for LMC.
}

For technical reasons, we also pair this assumption with a concentration property of the target: for some $\mf m \ge 0$ and $\alpha_1 \in [0,1]$,
\begin{align}\label{eq:tail}\tag{$\alpha_1\textsf{-tail}$}
    \pi\{\norm\cdot \ge \mf m + \lambda\} \le 2\exp\bigl\{-\bigl(\frac{\lambda}{C_{\msf{tail}}}\bigr)^{\alpha_1}\bigr\}\,, \qquad\text{for all}~\lambda \ge 0\,.
\end{align}
The parameters $\alpha_0$ and $\alpha_1$ are analogous to the parameter $\alpha$ in the LO inequality. Canonical examples of potentials satisfying~\eqref{eq:mlsi} and~\eqref{eq:tail} include $\norm{x}^\alpha$ with $\alpha_0 = \alpha_1 = \alpha$ and $C_{\msf{MLSI}}=O(1)$ for fixed $\alpha > 1$, as well as $\sqrt{1+\norm{x}^2}$, which satisfies~\eqref{eq:mlsi} with $\alpha_0 = -O(\frac{1}{\log d}), C_{\msf{MLSI}} = O(\log d)$, while $\alpha_1 = 1$. In both cases, $C_{\msf{tail}}=O(d^{O(1)})$. We refer to~\cite{erdogduhosseinzadeh2021tailgrowth} and the examples in Section~\ref{scn:examples} for further discussion. 

Similarly to Theorem~\ref{thm:lo_renyi}, we can prove a quantitative continuous-time convergence rate for the Langevin diffusion~\eqref{eq:langevin} under an MLSI\@. The proof is deferred to Section~\ref{scn:mlsi}.

\begin{theorem}\label{thm:mlsi_cont_time}
    Suppose that $\pi$ satisfies the conditions~\eqref{eq:mlsi} and~\eqref{eq:tail}, and assume that $\varepsilon^{-1}, \mf m, C_{\msf{MLSI}} \ge 1$ and that $\mf m, C_{\msf{tail}}, \eu R_{2q}(\pi_0 \mmid \pi) \le d^{O(1)}$. Let ${(\pi_t)}_{t\ge 0}$ denote the law of the continuous-time Langevin diffusion~\eqref{eq:langevin}.
    Then, it holds that $\eu R_q(\pi_T \mmid \pi) \le \varepsilon$ if
    \begin{align*}
        T \ge \Omega\Bigl(qC_{\msf{MLSI}}^2 \, \bigl(\mf m + qC_{\msf{tail}} \, {\eu R_{2q}(\pi_0 \mmid \pi)}^{1/\alpha_1}\bigr)^{2-\alpha_0} \polylog 
        \frac{d\, {\eu R_{q}(\pi_0 \mmid \pi)}}{\varepsilon}\Bigr)\,.
    \end{align*}
\end{theorem}

We remark that when $\alpha_0=\alpha_1=\alpha$, 
the dependence on the R\'enyi divergence at initialization in Theorems~\ref{thm:lo_renyi} and \ref{thm:mlsi_cont_time} match up to a logarithmic factor, and hence LOI and MSLI provide similar results in continuous time. However, as we discuss in Section~\ref{scn:examples}, 
MLSI is useful for treating certain examples in which the LOI constant $\CLOI$ may be dimension-dependent whereas $C_{\msf{MLSI}}$ is not.

\section{Main results on Langevin Monte Carlo}\label{scn:results}

In this section, we present our main results on the R\'enyi convergence of LMC. Denoting the step size with $h > 0$,
the LMC algorithm is defined by the iteration
\begin{align}\label{eq:lmc}\tag{$\msf{LMC}$}
    x_{(k+1)h}
    &= x_{kh} - h\nabla V(x_{kh}) + \sqrt{2h} \, \xi_k\,, \qquad k\in\N\,,
\end{align}
where ${(\xi_k)}_{k\in\N}$ is a sequence of i.i.d.\ standard Gaussian random variables. Here, the indexing of the LMC iterates is chosen so that the iterate $x_{kh}$ is comparable to the continuous-time diffusion~\eqref{eq:langevin} at time $kh$.
We let $\mu_{kh}$ denote the law of $x_{kh}$.

We also clarify the meaning of our theorem statements: we shall write that, provided $h = \widetilde \Theta(h_0)$, then the R\'enyi divergence is at most $\varepsilon$ after $N = \widetilde \Theta(N_0)$ iterations. Narrowly construed, this means that there exist choices of $h$, $N$, satisfying $h = \widetilde \Theta(h_0)$ and $N = \widetilde \Theta(N_0)$, such that the theorem statement is valid with these choices of $h$ and $N$. However, an inspection of the proofs show that our results are more robust in that they will hold for some range of values of $N$.

\subsection{Log-Sobolev inequality and gradient Lipschitz case}

Our first result deals with the case where $\nabla V$ is Lipschitz and $\pi$ satisfies~\eqref{eq:lsi}.

\begin{theorem}\label{thm:lsi_result}
    Assume that $\pi$ satisfies~\eqref{eq:lsi} and that $\nabla V$ is $L$-Lipschitz; assume for simplicity that $C_{\msf{LSI}}, L \ge 1$ and $q \ge 3$. Let $\mu_{Nh}$ denote the law of the $N$-th iterate of~\ref{eq:lmc} with step size $h$ satisfying $0 < h < 1/(192q^2 C_{\msf{LSI}} L^2)$. Then, for all $N \ge N_0$, it holds that
    \begin{align*}
        \eu R_q(\mu_{Nh} \mmid \pi)
        &\le \exp\Bigl( - \frac{(N-N_0) h}{4C_{\msf{LSI}}} \Bigr) \, \eu R_2(\mu_0 \mmid \pi) + \widetilde O(dh q C_{\msf{LSI}} L^2)\,,
    \end{align*}
    and where $N_0 = \lceil \frac{2C_{\msf{LSI}}}{h} \ln(q-1)\rceil$.
    In particular, if we choose $h = \widetilde\Theta(\frac{\varepsilon}{dqC_{\msf{LSI}} L^2} \min(1, \frac{d}{q\varepsilon}))$, then
    \begin{align*}
        \eu R_q(\mu_{Nh} \mmid \pi) \le \varepsilon\,, \qquad\text{for all}~N \ge \widetilde\Omega\Bigl(\frac{dq C_{\msf{LSI}}^2 L^2 \log \eu R_2(\mu_0 \mmid \pi)}{\varepsilon} \max\bigl\{1, \frac{q\varepsilon}{d}\bigr\}\Bigr)\,.
    \end{align*}
\end{theorem}

The statement ``for all $N \ge \widetilde \Omega(N_0)$'' should be interpreted similarly as the discussion following Theorem~\ref{thm:pi_lsi_renyi}.
The comparison of Theorem~\ref{thm:lsi_result} with~\cite{vempalawibisono2019ula, ganeshtalwar2020renyi, erdogduhosseinzadehzhang2022chisquare} is summarized as Table~\ref{table:lsi_result_comparison}.
Since our guarantee is stable with respect to the number of iterations $N$, we can let $N\to\infty$ and obtain an estimate on the asymptotic bias of~\eqref{eq:lmc} in R\'enyi divergence; this answers an open question of~\cite{vempalawibisono2019ula}.

\begin{corollary}\label{cor:renyi_bias}
    Assume that $\pi$ satisfies~\eqref{eq:lsi} and that $\nabla V$ is $L$-Lipschitz; assume for simplicity that $C_{\msf{LSI}}, L \ge 1$. Let $\mu_\infty^{(h)}$ denote the stationary distribution of~\ref{eq:lmc} with step size $h$ satisfying $0 < h < 1/(192q^2 C_{\msf{LSI}} L^2)$. Then,
    \begin{align*}
        \eu R_q(\mu_\infty^{(h)} \mmid \pi)
        &= \widetilde O(dh q C_{\msf{LSI}} L^2)\,.
    \end{align*}
\end{corollary}

\begin{table}[h]
\centering
\begin{tabular}{c c cc c} 
 \hline
 Source & Assumption & Metric & Complexity & Stable? \\
 \hline
 \rule{0pt}{1\normalbaselineskip} \cite{vempalawibisono2019ula} & \eqref{eq:lsi} & KL ($q=1$) & $dC_{\msf{LSI}}^2 L^2/\varepsilon$ & \cmark \\[0.25em]
 \cite{ganeshtalwar2020renyi} & $C_{\msf{SC}}^{-1}$-SLC & R\'enyi & $dq^2 C_{\msf{SC}}^4 L^4/\varepsilon^2$ &\xmark \\[0.25em]
 \cite{erdogduhosseinzadehzhang2022chisquare} & $C_{\msf{SC}}^{-1}$-SLC & R\'enyi & $dq^4 C_{\msf{SC}}^4 L^4/\varepsilon$ & \xmark \\[0.25em]
 \textbf{Theorem~\ref{thm:lsi_result}} & \eqref{eq:lsi} & R\'enyi & $dq C_{\msf{LSI}}^2 L^2/\varepsilon$ & \cmark
\end{tabular}
\caption{We compare the guarantee of Theorem~\ref{thm:lsi_result} with prior results, omitting polylogarithmic factors. ``SLC'' refers to ``strongly log-concave'', and the last column refers to whether the bound is stable as the number of iterations of LMC tends to infinity. The complexity bound in the last row is stated for moderate values of $q$; when $q \gg d/\varepsilon$, then the dependence on $q$ becomes $\widetilde O(q^2)$.}\label{table:lsi_result_comparison}
\end{table}

The above bound characterizes the bias for the Euler--Maruyama scheme with respect to the
target invariant measure $\pi$ in R\'enyi divergence.
Although the bias of the Euler--Maruyama scheme has been studied extensively in the numerical
stochastic analysis literature, quantitative dependence on important parameters such as the dimension is largely absent. Corollary~\ref{cor:renyi_bias} is the first to establish this in a strong
distance measure such as R\'enyi divergence.

\subsection{Log-concave and gradient Lipschitz case}

Extending the techniques of Theorem~\ref{thm:lsi_result}, we next give a result for the log-concave (which implies~\eqref{eq:pi}) and gradient Lipschitz case.

\begin{theorem}\label{thm:log_concave_result}
    Assume that $\pi$ is log-concave (and hence satisfies~\eqref{eq:pi}) and that $\nabla V$ is $L$-Lipschitz.
    For simplicity, assume that $V$ is minimized at $0$.
    Let $\mu_{Nh}$ denote the $N$-th iterate of LMC with step size $h$ satisfying $h = \widetilde \Theta(\frac{\varepsilon}{dq^2 C_{\msf{PI}} L^2} \min\{1, \frac{1}{q\varepsilon}, \frac{dC_{\msf{PI}}}{\varepsilon L}\})$ and initialized at $\mu_0 = \normal(0, L^{-1} I_d)$. Then,
    \begin{align*}
        \eu R_q(\mu_{Nh} \mmid \pi) \le \varepsilon \qquad\text{after}~N
    &= \widetilde \Theta\Bigl( \frac{d^2 q^3 C_{\msf{PI}}^2 L^2}{\varepsilon} \max\bigl\{1, q\varepsilon, \frac{\varepsilon L}{dC_{\msf{PI}}}\bigr\}\Bigr)~\text{iterations}\,.
    \end{align*}
\end{theorem}

In Table~\ref{table:log_concave_comparison}, we compare the above guarantee with the prior works~\cite{durmusmajewski2019lmcconvex, dwivedi2019log, chenetal2020hmc, DalKarRio22NonStrongly}. Theorem~\ref{thm:log_concave_result} is the best known convergence guarantee for algorithms based on discretizing the overdamped or the underdamped Langevin diffusions, only beaten by the result for modified MALA~\cite[Section 3.2]{dwivedi2019log}. In  this context, our result reads $\widetilde O(d^2/\varepsilon^2)$ whereas the result for modified MALA is $\widetilde O(d^2/\varepsilon^{3/2})$. Moreover, our result is given in the strongest metric (R\'enyi divergence).

\begin{table}[ht]
\centering
\begin{tabular}{cccc} 
 \hline
 Source & Algorithm & Metric & Complexity \\
 \hline
 \rule{0pt}{1\normalbaselineskip}
 \cite{durmusmajewski2019lmcconvex} & averaged \ref{eq:lmc} & $\sqrt{\text{KL}}$ & $d^2/\varepsilon^4$ \\
 \cite{dwivedi2019log, chenetal2020hmc} & modified MALA & TV & $d^2/\varepsilon^{3/2}$ \\
 \cite{DalKarRio22NonStrongly} & modified \ref{eq:lmc} & $W_1$ & $d^2 /\varepsilon^4$ \\
 \cite{DalKarRio22NonStrongly} & modified \ref{eq:lmc} & $W_2$ & $d^2 /\varepsilon^6$ \\
 \cite{DalKarRio22NonStrongly} & modified ULMC & $W_1$ & $d^2 /\varepsilon^3$ \\
 \cite{DalKarRio22NonStrongly} & modified ULMC & $W_2$ & $d^2 /\varepsilon^5$ \\
 \textbf{Theorem~\ref{thm:log_concave_result}} & \ref{eq:lmc} & $\sqrt{\text{R\'enyi}}$ & $d^2/\varepsilon^2$
\end{tabular}
\caption{We compare the best known convergence guarantees for sampling from an isotropic log-concave distribution with $C_{\msf{PI}}, L = O(1)$. MALA refers to the Metropolis-adjusted Langevin algorithm, whereas ULMC refers to underdamped Langevin Monte Carlo algorithm.}\label{table:log_concave_comparison}
\end{table}

\subsection{Lata\l{}a--Oleszkiewicz inequality and gradient H\"older case}\label{sec:lo_case}

Subsequently, we consider the general case of an LO inequality. We also assume weak smoothness for some $s \in (0, 1]$ and $L > 0$:
\begin{align}\label{eq:holder}\tag{$s\textsf{-H\"older}$}
    \norm{\nabla V(x) - \nabla V(y)} \le L \, \norm{x-y}^s \qquad\text{for all}~x,y \in \R^d\,.
\end{align}
We note that the LO order $\alpha$ and the H\"older exponent $s$ need to satisfy $s+1\ge\alpha$.
Extensions of our results can also be obtained under a mixed smoothness assumption, similar to that found in \cite{Ngu22WeaklySmooth, NguDanChe23LMCWeaklySmooth}, in which
  \begin{align*}
    \norm{\nabla V(x) - \nabla V(y)} \leq \sum_{i=1}^k L_i\, \norm{x-y}^{s_i}
  \end{align*}
  for constants $L_i \ge 0$, $s_i \in (0, 1]$.
However, we omit this extension for brevity.

\begin{theorem}\label{thm:main}
    Assume that the potential satisfies $\nabla V(0) = 0$,~\eqref{eq:lo} of order $\alpha$, and~\eqref{eq:holder}.
    For simplicity, assume that $\varepsilon^{-1}, \mf m, \CLOI, L, \eu R_2(\mu_0 \mmid \hat\pi) \ge 1$ and $q\ge 2$; here, $\mf m \deq \int \norm\cdot \, \D \pi$ and $\hat\pi$ is a slightly modified version of $\pi$ which is introduced in the analysis (Section~\ref{scn:proof_main}).
    Then, \ref{eq:lmc} with an appropriate step size (given in~\eqref{eq:final_step_size}) satisfies $\eu R_q(\mu_{Nh} \mmid \pi) \le \varepsilon$ after
    \begin{align*}
        N
        = \widetilde \Theta_s&\Bigl( \frac{dq^{1+2/s} \CLOI^{1+1/s} L^{2/s} \, {\eu R_{2q-1}(\mu_0 \mmid \pi)}^{(2/\alpha - 1) \, (1+1/s)}}{\varepsilon^{1/s}}
    \times
        \max\Bigl\{1, q^{1/s} \varepsilon^{1/s}, \frac{\mf m^s}{d}, \frac{{\eu R_2(\mu_0 \mmid \hat\pi)}^{s/2}}{d} \Bigr\}\Bigr)
    \end{align*}
    iterations.
    Here, $\widetilde \Theta_s(\cdot)$ hides polylogarithmic factors and constants depending only on $s$.
\end{theorem}

\noindent We now make a few remarks to simplify the rate. 
First, although initialization is more subtle in the non-log-concave case, it is reasonable to take $\eu R_2(\mu_0 \mmid \hat\pi), \eu R_{2q-1}(\mu_0 \mmid \pi) = \widetilde O(d)$; we defer a detailed discussion of initialization to Appendix~\ref{scn:initialization}. 
Next, it is also reasonable to assume\footnote{This holds for, e.g., the potentials $V(x) = \norm x^\alpha$ for all $\alpha \in [1,2]$.} $\mf m = O(d)$,
in which case the third term in the maximum will never dominate. Focusing on the dependence on the dimension and target accuracy, we therefore obtain the simplified rate $\widetilde O(d^{(2/\alpha) \, (1+1/s) - 1/s}/\varepsilon^{1/s})$; in particular, in the smooth ($s=1$) case, the rate is $\widetilde O(d^{4/\alpha - 1}/\varepsilon)$. Regarding prior works which handle a wide variety of growth rates and smoothness conditions for the potential, the closest to the present work is~\cite{erdogduhosseinzadeh2021tailgrowth}, which obtains a rate of $\widetilde O(d^{(2/\alpha + \one\{\alpha = 1\}) \, (1+1/s) - 1}/\varepsilon^{1/s})$ for potentials of tail growth $\alpha$ satisfying~\eqref{eq:holder}; note that our rate is strictly better as soon as $s < 1$ and avoids the jump in the rate at $\alpha = 1$. We emphasize, however, that despite the superficial similarity with~\cite{erdogduhosseinzadeh2021tailgrowth}, our result is the first one under a purely functional analytic condition on the target (together with weak smoothness). \medskip

\begin{remark}
    The case $\alpha=1$ yields the bound $\widetilde O(d^{2+1/s} q^{1+2/s}C_{\msf{PI}}^{1+1/s} L^{2/s}/\varepsilon^{1/s})$ for LMC under the Poincar\'e inequality and weak smoothness. In the case $\alpha = 2$ and $s = 1$ (LSI and smooth case), the rate estimate reduces to $\widetilde O(dq^3 C_{\msf{LSI}}^2 L^2/\varepsilon)$, which recovers the guarantee of Theorem~\ref{thm:lsi_result} up to the dependence on $q$.
\end{remark}

\subsection{Modified log-Sobolev inequality and gradient H\"older case}\label{sec:mlsi_case}

When the LO constant $\CLOI$ is dimension-dependent, Theorem~\ref{thm:main} may not give the sharpest rates. We therefore complement Theorem~\ref{thm:main} with a result assuming~\eqref{eq:mlsi}.

\begin{theorem}\label{thm:main_mlsi}
    Assume that the potential satisfies $\nabla V(0) = 0$,~\eqref{eq:mlsi} of order $\alpha_0$,~\eqref{eq:tail}, and~\eqref{eq:holder}.
    For simplicity, assume that $\varepsilon^{-1}, \mf m, C_{\msf{MLSI}}, C_{\msf{tail}}, L, \eu R_2(\mu_0 \mmid \hat\pi) \ge 1$, $q\ge 2$, and $\mf m, C_{\msf{tail}}, \eu R_2(\pi_0 \mmid \pi) \le d^{O(1)}$; here, $\hat\pi$ is a slightly modified version of $\pi$ which is introduced in the analysis (Section~\ref{scn:proof_main}).
    Then, \ref{eq:lmc} with an appropriate step size (given in~\eqref{eq:final_step_size_mlsi}) satisfies $\eu R_q(\mu_{Nh} \mmid \pi) \le \varepsilon$ after
    \begin{align*}
        N
        &= \widetilde \Theta\Bigl( \frac{d \, {
        \eu R_{2q}(\mu_0 \mmid \pi)}^{(2-\alpha_0) \, (1+1/s)/\alpha_1}}{\varepsilon^{1/s}}
    \times 
    \max\Bigl\{1, \varepsilon^{1/s}, \frac{\mf m^s}{d}, \frac{{\eu R_2(\mu_0 \mmid \hat\pi)}^{s/2}}{d}, \bigl( \frac{\mf m}{\eu R_{2q}(\mu_0 \mmid \pi)^{1/\alpha_1}} \bigr)^{(2-\alpha_0)/s} \Bigr\}\Bigr)
    \end{align*}
    iterations. Here, the $\widetilde \Theta(\cdot)$ notation hides polylogarithmic factors as well as constants depending on $\alpha_0$, $\alpha_1$, $q$, $s$, $C_{\msf{MLSI}}$, $C_{\msf{tail}}$, and $L$; a more precise statement is given in Section~\ref{scn:mlsi}.
\end{theorem}

For potentials of tail growth $\alpha \in (1, 2]$, 
we can suppose that~\eqref{eq:mlsi} and~\eqref{eq:tail} 
are satisfied with $\alpha_0 = \alpha_1 = \alpha$, 
where we take $\mf m = O(d^{1/\alpha})$. 
Also, assuming ${\eu R_2(\mu_0 \mmid \hat\pi)}, {\eu R_{2q}(\mu_0 \mmid \pi)} = O(d)$, 
the rate is then simplified to 
$\widetilde O(d^{(2/\alpha) \, (1+1/s) - 1/s}/\varepsilon^{1/s})$ 
as before.
As discussed in the next section, 
the case $\alpha = 1$ is special and~\eqref{eq:mlsi} may not hold with $\alpha_0 = \alpha$. \medskip

\begin{remark}
    A number of recent works~\cite{durmusmajewski2019lmcconvex, chatterjietal2020lmcwithoutsmoothness, liangchen2021proximal, Lehec23LMCNonSmooth, NguDanChe23LMCWeaklySmooth} consider non-smooth and mixed-smooth potentials. By incorporating Gaussian smoothing, it seems possible to extend our techniques to cover these settings, but we do not pursue this direction here.
\end{remark}

\section{Applications}\label{scn:applications}

\subsection{Error of the Euler--Maruyama scheme}\label{scn:discretization_err}

Our results in Sections~\ref{sec:lo_case} and~\ref{sec:mlsi_case} rely on the following result, which bounds the error in R\'enyi divergence of the Euler--Maruyama discretization scheme---i.e.,~\eqref{eq:lmc}---for the continuous-time Langevin SDE \eqref{eq:langevin}, under the general assumption of \eqref{eq:holder}.

\begin{theorem}\label{thm:main_disc_bd}
    Let ${(\mu_t)}_{t\ge 0}$ denote the law of the interpolated process~\eqref{eq:interpolated_lmc} and let ${(\pi_t)}_{t\ge 0}$ denote the law of the continuous-time Langevin diffusion~\eqref{eq:langevin}, both initialized at $\mu_0$. Assume that $\nabla V$ satisfies $\nabla V(0) = 0$ and~\eqref{eq:holder}.
    For simplicity, assume $\varepsilon^{-1}, \mf m, L, T, \eu R_2(\mu_0 \mmid \hat\pi) \ge 1$ and $q\ge 2$.
    If the step size $h$ satisfies
    \begin{align*}
        h = O_s\Bigl( \frac{\varepsilon^{1/s}}{dq^{1/s} L^{2/s} T^{1/s}} \min\Bigl\{1, \frac{d}{\mf m^s}, \frac{d}{{\eu R_2(\mu_0 \mmid \hat\pi)}^{s/2}} \Bigr\}\Bigr)\,,
    \end{align*}
    where the notation $O_s$ hides constants depending on $s$, then for $T \deq Nh$,
    \begin{align*}
        \eu R_q(\mu_T \mmid \pi_T)
        &\le \varepsilon\,.
    \end{align*}
\end{theorem}
Although error bounds for the Euler--Maruyama scheme have been studied extensively in numerical
stochastic analysis, the focus on obtaining explicit dependence, especially with regards to the dimension, is more recent. Theorem~\ref{thm:main_disc_bd} is the first such guarantee which holds in R\'enyi divergence. Note that the requirement on the step size depends on the smoothness and $\mf m$ of the target, but the result does not require any assumption on its isoperimetric properties.

\subsection{R\'enyi divergence and warm starts} \label{sec:warm-start}

As a corollary of our main results, we show that if we can output samples from an approximate distribution $\mu_N$ such that the R\'enyi divergence to $\pi$ is bounded, this can be used as a "warm start" for various higher accuracy Markov chains.
  Such warm start conditions are ubiquitous in sampling ~\cite{chen2018fast, chewi2021optimal, luwang2022zigzag}, and typically allow the user to combine a robust but slow method (such as LMC), to speed up an initialization-dependent method (such as MALA). The combined algorithm then has a lower sample complexity than each individual part run independently. 

  An important characteristic of a warm start $\mu$ to a target $\pi$ is that, provided $\pi$ assigns a small probability to an event $S$, then the probability of $S$ under $\mu$ must also be suitably small. However, common distance measures
  such as Wasserstein, total variation, or $\mathsf{KL}$ divergence do not exhibit this
  property.
  In contrast, we show that a small R\'enyi divergence of any order $q > 1$
  implies this condition.

  \begin{proposition}
    Let $\{\mu_k\}_{k\geq 0}$ denote the law of a lazy and reversible Markov chain (see~\cite{lovaszsimonovits1993randomwalk} for a formal definition) with invariant measure $\pi$ and
    $s$-conductance $\Phi(s)$, initialized at $\mu_0$. For some $q > 1$ and $\delta_0>2$, let the initialization
    satisfy $\eu R_q(\mu_0\mmid\pi)\leq \log(\delta_0/2)$.
    Then, 
    \begin{align*}
      N = \frac{2\log(\delta_0\, (2\delta_0)^q/\epsilon^q)}{\Phi\big(({\epsilon}/{2\delta_0})^q\big)}\ \text{ iterations}
    \end{align*}
    suffices to achieve $\|\mu_N - \pi\|_{\mathrm{TV}} \leq \epsilon$.
  \end{proposition}
  \begin{proof}
  Let $\beta(s) = \sup_{B\in \mc B(\mbb R^d)}\{|\mu_0(B) -\pi(B)|: \pi(B)\leq s\}$.
  Using~\cite[Corollary 1.6]{lovaszsimonovits1993randomwalk}, it holds that
  \begin{align}\label{eq:lovasz-tv-conductance}
    \|\mu_N - \pi\|_{\text{TV}} \leq \beta(s) + \frac{\beta(s)}{s}\, \Bigl(1- \frac{\Phi(s)}{2} \Bigr)^{N}\,.
  \end{align}
Also, for any two probability distributions $\mu$ and $\pi$, we can write for any event $B$
  \begin{align*}
    \abs{\mu(B) -\pi(B)}
    &= \Bigl\lvert \int \ind{B}\, \frac{\dee \mu}{\dee \pi}\, \dee \pi -
    \int \ind{B}\,\dee \pi\Bigr\rvert\\
    &= \Bigl\lvert\int \ind{B}\,\Bigl(\frac{\dee \mu}{\dee \pi}-1 \Bigr)\,\dee \pi
    \Bigr\rvert\\
    &\leq \Bigl(\int \ind{B}\,\dee \pi \Bigr)^{1/q}\,
    \Bigl(\int\Big|\frac{\dee \mu}{\dee \pi}-1 \Big|^q\,\dee \pi \Bigr)^{1/q}\\
    &= \pi(B)^{1/q}\, \chi_q(\mu \mmid \pi)\,.
  \end{align*}
  Thus we have $\beta(s) \leq s^{1/q}\,\chi_q(\mu_0\mmid \pi)$. Plugging this into \eqref{eq:lovasz-tv-conductance} yields
  \begin{align*}
    \|\mu_N - \pi\|_{\text{TV}} \leq s^{1/q}\,\chi_q(\mu_0\mmid\pi) + s^{-(1-1/q)}\,\chi_q(\mu_0\mmid \pi)\, e^{-N\Phi(s)/2}
  \end{align*}
  where in the last line we used $1-x \leq e^{-x}$. Therefore, when
  $\chi_q(\mu_0 \mmid \pi) \leq \delta_0$, we can choose
  $$s =
  \Big(\frac{\epsilon}{2\delta_0}\Big)^q\ \text{ and }\
  N = \frac{2\log(\delta_0\, (2\delta_0)^q/\epsilon^q)}{\Phi(s)}
  $$
  to obtain $ \|\mu_N - \pi\|_{\text{TV}}\leq\epsilon$.

  Next, we state the subsequent lemma.

\begin{lemma}
  For any $q > 1$,
  \begin{align*}
      \chi_q(\mu \mmid \pi)
      &\le 2\exp \eu R_q(\mu \mmid \pi)\,.
  \end{align*}
\end{lemma}
\begin{proof}
  Denote the density ratio with $\rho = \dee \mu / \dee \pi$.
From the elementary inequality $\abs{x+y}^q \le 2^{q-1}\,(\abs x^q + \abs y^q)$, we deduce that
\begin{align*}
    \abs{\rho-1}^q
    &\le 2^{q-1}\,(1+\rho^q)
\end{align*}
and upon integrating w.r.t.\ $\pi$,
\begin{align*}
    {\chi_q(\mu \mmid \pi)}^q
    &\le 2^{q-1} \,\Bigl(1 + \int \rho^q \, \dee \pi\Bigr)
    \le 2^q \int \rho^q\,\dee \pi
    = 2^q \exp\bigl((q-1)\,\eu R_q(\mu \mmid \pi)\bigr)\,.
\end{align*}
Taking the $1/q$-th power finishes the proof.
\end{proof}

 Using the above lemma and taking $ 2\exp
  \eu R_q(\mu_0\mmid \pi) \le \delta_0$ is sufficient to guarantee the initial $\chi_q(\mu_0 \mmid \pi) \leq \delta_0$, which concludes the proof.
\end{proof}

One important application of this is found, for example, in the analysis of the zig-zag sampler \cite[Corollary 1.4]{luwang2022zigzag}, which directly uses the R\'enyi convergence of LMC to provide warm starts for the algorithm.

\subsection{Examples}\label{scn:examples}

In contrast to prior analysis, our results are applicable to targets satisfying a Poincar\'e inequality (or more generally LOI). This embraces an extensive class of targets, such as all log-concave measures with finite second moment and perturbations thereof~\cite{kls1995, klartag2022bourgain}. The convergence under Poincar\'e inequalities of Fokker--Planck equations in $L^2$ (equivalent to $2$-R\'enyi) has been previously studied in functional analysis and sampling~\cite{markowich2000trend, pavliotis2014stochastic}, and is of substantial interest to the community. Many practical applications can also be found in Bayesian regression (e.g., Bayesian logistic regression with a heavy-tailed prior~\cite{gorham2019measuring}).

While the class of potentials satisfying LOI/MLSI is comparatively not as well-understood, we see this condition as a natural interpolation between PI and LSI. 
Since LOIs are also stable under bounded perturbation, 
this essentially permits most potentials with asymptotic tail growth $\norm{x}^\alpha$, $\alpha \in [1,2]$. Our conditions are also independent of the dissipativity assumptions required in prior works
therefore our results  are significantly more general.

The following examples illustrate the applicability of our bounds. Where not explicitly specified, we are comparing against \cite{erdogduhosseinzadeh2021tailgrowth} which obtains bounds in most of these regimes. 

\begin{example}[{tail growth $\alpha \in (1, 2]$}]\label{ex:tail-growth}
    Consider the target $\pi_\alpha(x) \propto \exp(-\norm x^\alpha)$ for $\alpha \in (1,2]$, which satisfies~\eqref{eq:lo} of order $\alpha$ and~\eqref{eq:holder} with $s = \alpha - 1$. Since $\pi_\alpha$ satisfies~\eqref{eq:pi} with $C_{\msf{PI}} = \Theta(d^{2/\alpha-1})$~\cite{bobkov2003sphericallysymmetric}, then Theorem~\ref{thm:main} does not yield a good result.
    Previous work showed that $\pi_\alpha$ satisfies~\eqref{eq:mlsi} and~\eqref{eq:tail} with order $\alpha_0 = \alpha_1 = \alpha$, obtaining the complexity $\widetilde O(d^{(3-\alpha)/(\alpha - 1)}/\varepsilon^{1/(\alpha - 1)})$ to achieve $\varepsilon$-accuracy in KL divergence for this target. 
    From Theorem~\ref{thm:main_mlsi}, we have improved this rate estimate to $\widetilde O({(d/\varepsilon)}^{1/(\alpha - 1)})$ in R\'enyi divergence. Here, it is a straightforward calculation to verify that $\mf m = O(d^{1/\alpha})$, $C_{\msf{MLSI}} = O(1)$, $C_{\msf{tail}}=O(d^{O(1)})$, and $L = O(1)$, for fixed $\alpha > 1$. 
    By Lemmas~\ref{lem:initialization_general} and~\ref{lem:initialization_pseudo}, we take $ \eu R_{2}(\mu_0 \mmid \hat\pi) = \widetilde{O}(d)$ and $\eu R_{2q}(\mu_0 \mmid \pi) = \widetilde{O}(d)$. Substituting this all into Theorem~\ref{thm:main_mlsi}, we obtain the claimed rate, noting that the maximum over the four terms is $O(1)$ when $\varepsilon \leq 1$.
    Since~\eqref{eq:mlsi} is stable under bounded perturbations, the same rate estimate holds for the perturbed potential $V(x) = \norm x^\alpha + \cos{\norm x}$.
    
    Due to the use of the CKP inequality~\cite{bolleyvillani2005weightedpinsker},
    their KL bound yields $\widetilde O(d^{(5-\alpha)/(\alpha - 1)}/\varepsilon^{\alpha/(\alpha - 1)})$ complexity to reach $\varepsilon$ accuracy in the $W_\alpha^2$ metric. On the other hand, Theorem~\ref{thm:main_mlsi} together with the Poincar\'e inequality yields the complexity $\widetilde O(d^{2/(\alpha \, (\alpha - 1))}/\varepsilon^{1/(\alpha - 1)})$ to obtain $\varepsilon$ accuracy in the $W_2^2$ metric. Hence, we have both improved the rate in $W_\alpha$ and proven a new guarantee in $W_2$ which previously could not be reached at all.
\end{example}

\begin{example}[{tail growth $\alpha \in (1,2]$ for smooth potential}]
    Consider target $\pi_\alpha(x) \propto \exp(-{(1+\norm x^2)}^{\alpha/2})$, which satisfies~\eqref{eq:mlsi} of order $\alpha_0 = \alpha_1 = \alpha$ and~\eqref{eq:holder} with $s = 1$ (i.e., $\nabla V$ is Lipschitz). Here, all relevant constants are of the same order as in the previous example.
    Previous work had obtained the complexity $\widetilde O(d^{(4-\alpha)/\alpha}/\varepsilon)$ in KL divergence and $\widetilde O(d^{(4+\alpha)/\alpha}/\varepsilon^\alpha)$ in $W_\alpha^2$. From Theorem~\ref{thm:main_mlsi}, we have obtained the rate $\widetilde O(d^{(4-\alpha)/\alpha}/\varepsilon)$ in R\'enyi divergence and $\widetilde O(d^{(6-2\alpha)/\alpha}/\varepsilon)$ in $W_2^2$. As before, this rate is stable under suitable perturbations of the potential.
\end{example}

\begin{example}[{tail growth $\alpha = 1$ for smooth potential}]
    The case of $\alpha = 1$ is worth considering separately for comparison purposes.
    Consider the target $\pi_1(x) \propto \exp(-\sqrt{1+\norm x^2})$, which satisfies~\eqref{eq:holder} with $s = 1$ (i.e., $\nabla V$ is Lipschitz).
    In previous work, it was shown that $\pi_1$ satisfies~\eqref{eq:mlsi} with $\alpha_0 = -O(\frac{1}{\log d})$ and $C_{\msf{MLSI}} = O(\log d)$; also, $\pi_1$ satisfies~\eqref{eq:tail} with $\alpha_1 = 1$. Using this, they
    obtained the complexity $\widetilde O(d^5/\varepsilon)$ in KL divergence, whereas Theorem~\ref{thm:main_mlsi} implies the same rate in R\'enyi divergence. 
    We also remark that prior rates only held for sufficiently small perturbations (e.g., their analysis does not cover $V(x) = \norm x + \cos{\norm x}$) due to the need to preserve a dissipativity assumption, whereas our result has no such requirement. This highlights a benefit of working without dissipativity conditions.
    
    Here, Theorem~\ref{thm:log_concave_result} applies to $\pi_1$ with $C_{\msf{PI}} = O(d)$~\cite{bobkov2003sphericallysymmetric},
    $\mf m = O(d)$, $L = O(1)$, and $C_{\msf{tail}}=O(d^{O(1)})$,
    and yields a rate of $\widetilde O(d^4/\varepsilon)$ in R\'enyi divergence; in contrast,~\cite{durmusmajewski2019lmcconvex} yields a rate of $\widetilde O(d^3/\varepsilon^2)$ in KL divergence (started from a distribution with $W_2^2(\mu_0, \pi_1) = O(d^2)$) for averaged LMC, and~\cite{dwivedi2019log, chenetal2020hmc} yields a rate of $\widetilde O(d^{3.5}/\varepsilon^{0.75})$ in $\norm \cdot_{\msf{TV}}^2$ for modified MALA, although none of these rates is stable under perturbation.
\end{example}

\begin{example}[{tail growth $\alpha \in [1,2]$ for smooth product potential}]
    For $x\in\R^d$, let $\langle x\rangle_i \deq \sqrt{1+x_i^2}$.
    Consider the target $\pi_\alpha(x) \propto \exp(-\norm{\langle x\rangle}_\alpha^\alpha)$, which satisfies~\eqref{eq:lo} of order $\alpha$ with $C_{\msf{LO}(\alpha)} = O(1)$~\cite[see][]{latalaoleszkiewicz2000loineq}, \eqref{eq:holder} with $s = 1$ (i.e., $\nabla V$ is Lipschitz),  $\mf m = O(d^{1/\alpha})$, and $L = O(1)$.
    Prior work had implied a complexity of $\widetilde O(d^{(4-\alpha)/\alpha}/\varepsilon)$ in KL divergence and $\widetilde O(d^{(4+\alpha)/\alpha}/\varepsilon^\alpha)$ in $W_\alpha^2$ for $\alpha \in (1,2]$, and $\widetilde O(d^5/\varepsilon)$ in KL divergence when $\alpha = 1$. From Theorem~\ref{thm:main}, we have obtained the rate $\widetilde O(d^{(4-\alpha)/\alpha}/\varepsilon)$ in R\'enyi divergence and hence also $W_2^2$ for all $\alpha \in [1,2]$; in particular, there is no jump in the rate at $\alpha = 1$.
\end{example}

\begin{example}[LSI case with weakly smooth potential]
    We also compare the results when $\alpha = 2$ and $s \in (0, 1]$. In this case,~\cite{chatterjietal2020lmcwithoutsmoothness} obtained the rate $\widetilde O(d^{(2+s)/s}/\varepsilon^{1/s})$ in $\norm \cdot_{\msf{TV}}^2$ for strongly log-concave distributions, whereas~\cite{erdogduhosseinzadeh2021tailgrowth} obtained the rate $\widetilde O({(d/\varepsilon)}^{1/s})$ in KL divergence for perturbations of strongly log-concave distributions. In contrast, Theorem~\ref{thm:main} yields the rate $\widetilde O(d/\varepsilon^{1/s})$ in R\'enyi divergence under~\eqref{eq:lsi}. An example of such a potential is given by $V(x) = \frac{1}{2} \, \norm x^2 + \cos(\norm x^{1+s})$. 
\end{example}

\section{Technical overview}\label{scn:overview}

\subsection{Adapting the interpolation method to R\'enyi divergences}

In the proof of Theorem~\ref{thm:lsi_result}, 
we start with the
natural interpolation of~\eqref{eq:lmc}: for $t \in [kh, (k+1)h]$, let
\begin{align}\label{eq:interpolated_lmc}
    x_t
    &= x_{kh} - (t - kh) \, \nabla V(x_{kh}) + \sqrt 2 \, (B_t - B_{kh})\,,
\end{align}
where ${(B_t)}_{t\ge 0}$ is a standard Brownian motion, and let $\mu_t$ denote the law of $x_t$. The work~\cite{vempalawibisono2019ula} derives the following differential inequality for the KL divergence:
\begin{align}\label{eq:kl_diff_ineq}
    \partial_t \msf{KL}(\mu_t \mmid \pi)
    &\le -\frac{3}{4} \times \underbrace{4\E_\pi[\norm{\nabla\sqrt{\rho_t}}^2]}_{\text{Fisher information}} + \underbrace{\E[\norm{\nabla V(x_t) - \nabla V(x_{kh})}^2]}_{\text{discretization error}}\,, \qquad t \in [kh, (k+1)h]\,,
\end{align}
where we write $\rho_t  \deq  \frac{\D \mu_t}{\D \pi}$. 
This inequality is an analogue of the celebrated de Bruijn identity from information theory for the interpolated process.
Assuming that $\pi$ satisfies~\eqref{eq:lsi} and that $\nabla V$ is $L$-Lipschitz, the Fisher information upper bounds the KL divergence and the discretization error is shown to be of order $O(dh^2 L^2)$; this then yields a convergence guarantee in KL divergence.

The analogous differential inequality for the R\'enyi divergence is, for $t \in [kh, (k+1)h]$,
\begin{align}\label{eq:renyi_diff_ineq}
    \partial_t \eu R_q(\mu_t \mmid \pi)
    &\le -{\underbrace{\frac{3}{q} \,\frac{\E_\pi[\norm{\nabla(\rho_t^{q/2})}^2]}{\E_\pi(\rho_t^q)}}_{\text{R\'enyi Fisher information}}} + {\underbrace{q \, \frac{\E[\rho_t^{q-1}(x_t) \, \norm{\nabla V(x_t) - \nabla V(x_{kh})}^2]}{\E_\pi(\rho_t^q)}}_{\text{discretization error}}}\, ,
\end{align}
which we also provide a derivation in Proposition~\ref{prop:renyi_time_deriv_interpolation}. Note that the $q=1$ case of the above inequality formally corresponds to~\eqref{eq:kl_diff_ineq}. The R\'enyi Fisher information indeed upper bounds the R\'enyi divergence under an LSI (see e.g.~\cite[Lemma 5]{vempalawibisono2019ula}); however, the discretization term is now far trickier to control.

Write $\psi_t  \deq  \rho_t^{q-1}/\E_\pi(\rho_t^q)$. Observing that $\E \psi_t(x_t) = 1$, the discretization term can be written as an expectation under a \emph{change of measure}:
\begin{align*}
    \text{discretization error}
    &= q\, \widetilde\E[\norm{\nabla V(x_t) - \nabla V(x_{kh})}^2]\,,
\end{align*}
where $\widetilde \E$ is the expectation under the measure $\widetilde \Pr$ defined via $\frac{\D \widetilde\Pr}{\D \Pr} = \psi_t(x_t)$. Also, using the Lipschitzness of $\nabla V$, we obtain $\norm{\nabla V(x_t) - \nabla V(x_{kh})}^2 \le 2h^2 L^2 \, \norm{\nabla V(x_{kh})}^2 + 4L^2 \, \norm{B_t - B_{kh}}^2$. Hence, our task is to bound the expectation of these two terms under a complicated change of measure.

Towards that end, consider first the Brownian motion term. Using the Donsker--Varadhan variational principle, for any random variable $X$,
\begin{align*}
    \widetilde{\E} X
    &\le \msf{KL}(\widetilde\Pr \mmid \Pr) + \ln \E \exp X\,.
\end{align*}
Applying this to $X = c \, {(\norm{B_t - B_{kh}} - \E\norm{B_t - B_{kh}})}^2$ for a constant $c > 0$ to be chosen later, we can bound
\begin{align}\label{eq:brownian_donsker_varadhan}
    \begin{aligned}
    \!\!\!\!\!\!\!\!\widetilde \E[\norm{B_t - B_{kh}}^2]
    &\le 2\E[\norm{B_t - B_{kh}}^2] + \frac{2}{c} \, \widetilde\E X \\
    &\le 2\E[\norm{B_t - B_{kh}}^2] + \frac{2}{c} \, \bigl\{\msf{KL}(\widetilde \Pr \mmid \Pr) + \ln \E \exp\bigl(c \, {(\norm{B_t - B_{kh}} - \E\norm{B_t - B_{kh}})}^2\bigr)\bigr\}\,.
    \end{aligned}
\end{align}
Note that the first and third terms in the right-hand side of the above expression are expectations under the original measure $\Pr$, and can therefore be controlled; to ensure that the third term is bounded, we can take $c \asymp 1/h$. 
For the second term, a surprising calculation involving a judicious application of the LSI for $\pi$ (see~\eqref{eq:surprising_calc_1},~\eqref{eq:surprising_calc_2}, and~\eqref{eq:surprising_calc_3})
shows that it is bounded by $h$ times the R\'enyi Fisher information, and can therefore be absorbed into the first term of the differential inequality~\eqref{eq:renyi_diff_ineq} for $h$ sufficiently small.

The expectation of the drift term $\norm{\nabla V(x_{kh})}^2$ under the change of measure can also be handled via similar methods, but this can be bypassed via a duality principle for the Fisher information; see Lemma~\ref{lem:fisher_info_arg}. We also remark that na\"{\i}vely, this proof incurs a cubic dependence on $q$, but this can be sharpened via an argument based on hypercontractivity (Proposition~\ref{prop:hypercontractivity}).

In the above proof outline, the LSI for $\pi$ plays a crucial role in the arguments. In Theorem~\ref{thm:log_concave_result}, we show that the method can be somewhat extended to the case when $\pi$ does not satisfy an LSI, but is instead assumed to be (weakly) log-concave. In this case, we show that with an appropriate Gaussian initialization, the law $\mu_{kh}$ of the \emph{iterate} $x_{kh}$ of~\eqref{eq:lmc} satisfies an LSI, albeit with a constant which grows with the number of iterations (Lemma~\ref{lem:lsi_along_iterates_cvx}). In turn, this fact together with a suitable modification of the preceding proof strategy also allows us to obtain a convergence guarantee in this case (see Section~\ref{scn:proof_log_concave_result} for details).

\subsection{Controlling discretization error via Girsanov's theorem}\label{scn:discretization_overview}

In the general case of a weaker functional inequality and smoothness condition, the preceding arguments do not apply. Instead, we start with the weak triangle inequality for the R\'enyi divergence (when $q\geq 2$):
\begin{align*}
    \eu R_q(\mu_T \mmid \pi)
    &\lesssim \eu R_{2q}(\mu_T \mmid \pi_T) + \eu R_{2q-1}(\pi_T \mmid \pi)\,.
\end{align*}
Here, ${(\mu_t)}_{t\ge 0}$ is the law of the interpolated process~\eqref{eq:interpolated_lmc}, whereas ${(\pi_t)}_{t\ge 0}$ is the law of the continuous-time Langevin diffusion~\eqref{eq:langevin} initialized at a draw from $\mu_0$. The second term is handled via the continuous-time convergence results, either under the LO inequality (Theorem~\ref{thm:lo_renyi}) or under the MLSI (Theorem~\ref{thm:mlsi_cont_time}), and the crux of the proof is to control the first term (the discretization error).

First, the data processing inequality implies that $\eu R_{2q}(\mu_T \mmid \pi_T) \le \eu R_{2q}(P_T \mmid Q_T)$, where $P_T$ and $Q_T$ are measures on path space representing the laws of the trajectories (on the interval $[0, T]$) of the interpolated and diffusion processes respectively. Next, Girsanov's theorem provides a closed-form formula for the Radon-Nikodym derivative $\frac{\D P_T}{\D Q_T}$, which leads to the inequality
\begin{align*}
    \eu R_{2q}(P_T \mmid Q_T)
    &\le \frac{1}{2 \, (2q-1)} \ln \E \exp\Bigl(4q^2 \int_0^T \norm{\nabla V(z_t) - \nabla V(z_{\lfloor t/h \rfloor \, h})}^2 \, \D t\Bigr)\,,
\end{align*}
where ${(z_t)}_{t\ge 0}$ is the continuous-time Langevin diffusion~\eqref{eq:langevin}. 
The use of Girsanov's theorem for deriving quantitative estimates on the discretization error in this manner was likely first introduced in~\cite{dalalyantsybakov2012sparseregression} for the KL divergence, although the current application to R\'enyi divergences is closer to the calculation for MALA in~\cite{chewietal2021mala}. 
However, to the best our knowledge, this paper is the first to adapt the Girsanov technique to provide a complete R\'enyi convergence result for \ref{eq:lmc}.

Controlling the discretization error over an interval $[0, h]$ corresponding to a single iteration of \ref{eq:lmc} is straightforward using the tools of stochastic calculus, and was in fact carried out in~\cite{chewietal2021mala}. 
Extending this to the full time interval $[0, T]$ is more challenging; 
indeed, if we bound the discretization error on $[h, 2h]$ conditional on ${(z_t)}_{t\in [0,h]}$, then the resulting bound depends on $\norm{z_h}^2$, which prevents us from straightforwardly iterating the one-step discretization bound. 
To address this, let $\Delta_{s,t} \deq \int_s^t \norm{\nabla V(z_r) - \nabla V(z_{\lfloor r/h \rfloor h})}^2 \, \D r$. Then, the AM{--}GM inequality yields
\begin{align*}
    \E\exp\Bigl(4q^2 \sum_{k=0}^{N-1} \Delta_{kh,(k+1)h}\Bigr)
    &\le \frac{1}{N} \sum_{k=0}^{N-1} \E\exp(4q^2 N \Delta_{kh,(k+1)h}).
\end{align*}
In order to bound these exponential moments, however, we require $\norm{z_{kh}}$ to have sub-Gaussian tails for $k=0,1,\dotsc,N-1$.

Observe however that the stationary distribution $\pi$ may not have sub-Gaussian tails under our assumption of an LO inequality (indeed, in the Poincar\'e case, $\pi$ may only have subexponential tails). Nevertheless, if the initialization $\mu_0$ has sub-Gaussian tails, then for each $t\in [0,T]$ it may still be the case that $\pi_t$ has sub-Gaussian tails. This turns out to be true, but it is quite non-trivial to prove without any dissipativity conditions on the potential $V$, and therefore constitutes our primary technical challenge.

To overcome this challenge, we introduce a novel technique based on comparison of the diffusion~\eqref{eq:langevin} with an auxiliary Langevin diffusion ${(\hat\pi_t)}_{t\ge 0}$ corresponding to a modified stationary distribution $\hat\pi$. The distribution $\hat\pi$ is constructed to have sub-Gaussian tails. To transfer the sub-Gaussianity of $\hat\pi$ to $\pi_t$, we apply the following change of measure inequality: for probability measures $\mu$ and $\nu$, and any event $E \subseteq \R^d$,
\begin{align*}
    \mu(E)
    &= \nu(E) + \int \one_E \, \bigl(\frac{\D\mu}{\D\nu} - 1\bigr) \,\D \nu
    \le \nu(E) + \sqrt{\chi^2(\mu \mmid \nu) \, \nu(E)}\,,
\end{align*}
where the last inequality is the Cauchy--Schwarz inequality.
This simple inequality states that in order to control the probability of an event $E$ under a measure $\mu$ in terms of its probability under $\nu$, it suffices to control the chi-squared divergence between $\mu$ and $\nu$. Applying this to our context, we can establish sub-Gaussian tail bounds for $\pi_t$ if we can control the R\'enyi divergences $\eu R_2(\pi_t \mmid\hat\pi_t)$ and $\eu R_2(\hat\pi_t \mmid\hat\pi)$; the former is again controlled via Girsanov's theorem. We stress that the auxiliary process ${(\hat\pi_t)}_{t\ge 0}$ is introduced only for analysis purposes and does not affect the implementation of the algorithm.

The details of this strategy are carried out in Section~\ref{scn:proof_main}.

\section{Proofs and further technical details}\label{scn:proofs}

\subsection{Proof of Theorem~\ref{thm:lo_renyi}}\label{scn:lo_renyi_pf}

In this section, we prove Theorem~\ref{thm:lo_renyi} on the R\'enyi convergence of the continuous-time Langevin diffusion~\eqref{eq:langevin} under an LO inequality. Using capacity inequalities as an intermediary,~\cite{barthecattiauxroberto2006interpolated, gozlan2010poincare} established the equivalence of LO inequalities with other functional inequalities such as modified Sobolev inequalities. For our purposes, it is convenient to work with \emph{super Poincar\'e inequalities}, which were introduced in~\cite{wang2000superpoincare}.

We say that $\pi$ satisfies a super Poincar\'e inequality with function $\beta : \R_+ \to \R_+$ if for all smooth $f : \R^d\to\R$,
\begin{align}\label{eq:super_poincare}
    \E_\pi(f^2) \le \beta(s) \E_\pi[\norm{\nabla f}^2] + s \, {(\E_\pi\abs f)}^2 \qquad\text{for all}~s\ge 1\,.
\end{align}
For $\alpha \in [1,2]$, define the function $\beta_\alpha : \R_+ \to \R_+$ via
\begin{align*}
    \beta_\alpha(s)
    &\deq \frac{96\CLOI}{{\ln(\e + s)}^{2-2/\alpha}}\,.
\end{align*}
Then, it is known that~\eqref{eq:lo} with order $\alpha$ implies a super Poincar\'e inequality with function $\beta_\alpha$~\cite[see][Remark 5.16]{gozlan2010poincare}.
\medskip

\begin{proof}[Proof of Theorem~\ref{thm:lo_renyi}]
    From~\cite[Lemma 6]{vempalawibisono2019ula}, we have
    \begin{align*}
        \partial_t \eu R_q(\pi_t \mmid \pi)
        &= - \frac{4}{q} \, \frac{\E_\pi[\norm{\nabla(\rho_t^{q/2})}^2]}{\E_\pi(\rho_t^q)}\,,
    \end{align*}
    where $\rho_t \deq \frac{\D \pi_t}{\D \pi}$. Applying the super Poincar\'e inequality~\eqref{eq:super_poincare} with $f = \rho_t^{q/2}$ and $\beta = \beta_\alpha$ yields
\begin{equation*}
\begin{aligned}
    \mathbb{E}_\pi[\| \nabla (\rho_t^{q/2}) \|^2]
    &\geq 
        \frac{1}{\beta_\alpha(s)} \E_\pi(\rho_t^q)
        - \frac{s}{\beta_\alpha(s)}
        \, \{\E_\pi(\rho_t^{q/2})\}^2 \\
    &= 
        \frac{1}{\beta_\alpha(s)} \exp\{(q-1) \, \eu R_q(\pi_t \mmid \pi)\} 
        - \frac{s}{\beta_\alpha(s)} \exp\{(q-2)\, \eu R_{q/2}(\pi_t \mmid \pi)\} \,. 
\end{aligned}
\end{equation*}
Using the fact that $\eu R_{q/2} \leq \eu R_q$, we can further lower bound this by 
\begin{align*}
    \mathbb{E}_\pi[\| \nabla (\rho_t^{q/2}) \|^2] 
    &\geq 
        \frac{\exp\{(q-1) \,\eu R_q(\pi_t \mmid \pi)\}}{\beta_\alpha(s)}  \,
        \bigl( 1 - s \exp\{-\eu R_q(\pi_t \mmid \pi)\}
        \bigr) \\
    &= \frac{\E_\pi(\rho_t^q)}{\beta_\alpha(s)}  \,
        \bigl( 1 - s \exp\{-\eu R_q(\pi_t \mmid \pi)\}
        \bigr)\,.
\end{align*}

We now distinguish two cases.
If $\eu R_q(\pi_t \mmid \pi) \ge 1$, then we choose $s = \frac{1}{2} \exp\{\eu R_q(\pi_t \mmid \pi)\}$, yielding
\begin{align*}
    \partial_t \eu R_q(\pi_t \mmid \pi)
    &\le - \frac{2}{q \beta_\alpha(s)}
    = - \frac{{\ln(\e + \frac{1}{2} \exp \eu R_q(\pi_t \mmid \pi))}^{2-2/\alpha}}{48q \CLOI} \\
    &\le - \frac{1}{68q \CLOI} \, {\eu R_q(\pi_t \mmid \pi)}^{2-2/\alpha}\,.
\end{align*}
Otherwise, if $\eu R_q(\pi_t \mmid \pi) \le 1$, then we choose $s = 1$, yielding
\begin{align*}
    \partial_t \eu R_q(\pi_t \mmid \pi)
    &\le - \frac{4}{q \beta_\alpha(1)} \, \bigl( 1 - \exp\{-\eu R_q(\pi_t \mmid \pi)\}\bigr)
    \le -\frac{2}{q\beta_\alpha(1)} \, \eu R_q(\pi_t \mmid \pi) \\
    &\le - \frac{1}{68q \CLOI} \, \eu R_q(\pi_t \mmid \pi)\,,
\end{align*}
where we used the elementary inequality $1-\exp(-x) \ge x/2$ for $x \in [0,1]$.
\end{proof}

\subsection{Proof of Theorem~\ref{thm:lsi_result}}

Throughout this section, recall the notation $\rho_t \deq \frac{\D \mu_t}{\D \pi}$ and $\psi_t \deq \rho_t^{q-1}/\E_\pi(\rho_t^q)$.

We begin by proving the differential inequality~\eqref{eq:renyi_diff_ineq}. Although this has appeared in the previous works~\cite{vempalawibisono2019ula, erdogduhosseinzadehzhang2022chisquare}, we include the proofs for the sake of completeness.

\begin{proposition}\label{prop:interpolated_continuity_eq}
    Let ${(\mu_t)}_{t\ge 0}$ denote the law of the interpolation~\eqref{eq:interpolated_lmc} of LMC\@. Then, for $t \in [kh, (k+1)h]$,
    \begin{align*}
        \partial_t \mu_t
        &= \divergence\bigl(\bigl\{\nabla \ln \frac{\D\mu_t}{\D\pi} + \E[\nabla V(x_{kh}) - \nabla V(x_t) \mid x_t = \cdot]\bigr\} \, \mu_t \bigr)\,.
    \end{align*}
\end{proposition}
\begin{proof}
    For $s,t\in\R_+$, let $\mu_{t\mid s}(\cdot \mid x_s)$ denote the conditional law of $x_t$ given $x_s$, and let $\mu_{s,t}$ denote the joint law of $(x_s, x_t)$. Conditioned on $x_{kh}$, the Fokker-Planck equation for the interpolation~\eqref{eq:interpolated_lmc} takes the form
    \begin{align*}
        \partial_t \mu_{t\mid kh}(\cdot \mid x_{kh})
        &= \Delta \mu_{t \mid kh}(\cdot \mid x_{kh}) + \divergence\bigl( \nabla V(x_{kh}) \, \mu_{t\mid kh}(\cdot \mid x_{kh})\bigr)\,.
    \end{align*}
    Taking the expectation over $x_{kh}$ yields
    \begin{align*}
        \partial_t \mu_t
        &= \Delta \mu_t + \divergence(\nabla V \, \mu_t) + \int \divergence\bigl(\{\nabla V(x_{kh}) - \nabla V(\cdot)\} \, \mu_{t\mid kh}(\cdot \mid x_{kh})\bigr) \, \D \mu_{kh}(x_{kh}) \\
        &= \divergence\bigl(\nabla \ln\frac{\D\mu_t}{\D\pi} \, \mu_t\bigr) + \divergence\Bigl(\bigl(\int \{\nabla V(x_{kh}) - \nabla V(\cdot)\} \, \D \mu_{kh \mid t}(x_{kh} \mid \cdot)\bigr) \, \mu_t(\cdot) \Bigr) \\
        &= \divergence\bigl(\nabla \ln\frac{\D\mu_t}{\D\pi} \, \mu_t\bigr) + \divergence\bigl( \{\E[\nabla V(x_{kh}) \mid x_t = \cdot] - \nabla V\} \, \mu_t \bigr)\,.
    \end{align*}
    Combining the two terms yields the result.
\end{proof}

\begin{proposition}\label{prop:renyi_time_deriv_interpolation}
    Let ${(\mu_t)}_{t\ge 0}$ denote the law of the interpolation~\eqref{eq:interpolated_lmc} of LMC\@. Also, let $\rho_t  \deq  \frac{\D\mu_t}{\D\pi}$ and $\psi_t  \deq  \rho_t^{q-1}/\E_\pi(\rho_t^q)$. Then, for $t \in [kh, (k+1)h]$,
    \begin{align*}
        \partial_t \eu R_q(\mu_t \mmid \pi)
        &\le -\frac{3}{q} \,\frac{\E_\pi[\norm{\nabla(\rho_t^{q/2})}^2]}{\E_\pi(\rho_t^q)} + q \E[\psi_t(x_t) \, \norm{\nabla V(x_t) - \nabla V(x_{kh})}^2]\,.
    \end{align*}
\end{proposition}
\begin{proof}
For brevity, in this proof we write $\Gamma_t \deq \E[\nabla V(x_{kh}) \mid x_t = \cdot] - \nabla V$. Elementary calculus together with Proposition~\ref{prop:interpolated_continuity_eq} yields
\begin{align*}
    \partial_t \eu R_q(\mu_t \mmid \pi)
    &= \frac{q}{(q-1) \E_\pi(\rho_t^q)} \int \bigl(\frac{\D\mu_t}{\D\pi}\bigr)^{q-1} \, \partial_t \mu_t \\
    &= \frac{q}{(q-1) \E_\pi(\rho_t^q)} \int \rho_t^{q-1} \divergence(\{\nabla \ln \rho_t + \Gamma_t\} \, \mu_t) \\
    &= -\frac{q}{(q-1) \E_\pi(\rho_t^q)} \int \langle \nabla (\rho_t^{q-1}), \nabla \ln \rho_t + \Gamma_t\rangle \, \D \mu_t \\
    &= - \frac{1}{\E_\pi(\rho_t^q)} \, \bigl\{ \frac{4}{q}\E_\pi[\norm{\nabla (\rho_t^{q/2})}^2] + 2 \E_{\mu_t}[\rho_t^{q/2-1} \, \langle \nabla(\rho_t^{q/2}), \Gamma_t \rangle] \bigr\}\,.
\end{align*}
For the second term, Young's inequality implies
\begin{align*}
    &-\E_{\mu_t}[\rho_t^{q/2-1} \, \langle \nabla (\rho_t^{q/2}), \Gamma_t \rangle] \\
    &\qquad = -\iint \rho_t^{q/2-1}(x_t) \, \langle \nabla(\rho_t^{q/2})(x_t), \nabla V(x_{kh}) - \nabla V(x_t) \rangle \, \mu_{kh \mid t}(\D x_{kh} \mid x_t) \, \mu_t(\D x_t) \\
    &\qquad = -\iint \rho_t^{q/2-1}(x_t) \, \langle \nabla(\rho_t^{q/2})(x_t), \nabla V(x_{kh}) - \nabla V(x_t) \rangle \, \mu_{kh, t}(\D x_{kh}, \D x_t) \\
    &\qquad = -\E[\rho_t^{q/2-1}(x_t) \, \langle \nabla(\rho_t^{q/2})(x_t), \nabla V(x_{kh}) - \nabla V(x_t) \rangle] \\
    &\qquad \le \frac{1}{2q} \E_\pi[\norm{\nabla (\rho_t^{q/2})}^2] + \frac{q}{2} \E[\rho_t^{q-1}(x_t) \, \norm{\nabla V(x_{kh}) - \nabla V(x_t)}^2]\,.
\end{align*}
Substituting this into the previous expression completes the proof.
\end{proof}

\begin{lemma}\label{lem:brownian_variational_arg}
Assume that $\nabla V$ is $L$-Lipschitz, and that $\pi$ satisfies~\eqref{eq:lsi}. Then, define $\rho_t \deq \frac{\D \mu_t}{\D \pi}, \psi_t \deq \rho_t^{q-1}/\E_\pi(\rho_t^q)$, and let $(B_t)_{t \geq 0}$ be a standard Brownian motion. For $t \in [kh, (k+1)h)$,
\begin{align*}
    \E[\psi_t(x_t) \, \norm{B_t - B_{kh}}^2]
    &\le 14d \, (t-kh) + 32hC_{\msf{LSI}} \, \frac{\E_\pi[\norm{\nabla(\rho_t^{q/2})}^2]}{\E_\pi(\rho_t^q)}\,.
\end{align*}
\end{lemma}

\begin{proof}
Applying the Donsker--Varadhan variational principle as in~\eqref{eq:brownian_donsker_varadhan}, we obtain
\begin{align*}
    &\E[\psi_t(x_t) \, \norm{B_t - B_{kh}}^2] \\
    &\qquad \le 2\E[\norm{B_t - B_{kh}}^2] + \frac{2}{c} \, \bigl\{\msf{KL}(\widetilde \Pr \mmid \Pr) + \ln \E \exp\bigl(c \, {(\norm{B_t - B_{kh}} - \E\norm{B_t - B_{kh}})}^2\bigr)\bigr\} \\
    &\qquad \le 2d \, (t-kh) + \frac{2}{c} \, \bigl\{\msf{KL}(\widetilde \Pr \mmid \Pr) + \ln \E \exp\bigl(c \, {(\norm{B_t - B_{kh}} - \E\norm{B_t - B_{kh}})}^2\bigr)\bigr\}\,,
\end{align*}
where $\frac{\D\widetilde\Pr}{\D\Pr} = \psi_t(x_t)$. Due to Gaussian concentration, if we set $c =\frac{1}{8 \, (t-kh)}$, then
\begin{align*}
    \E \exp\frac{{(\norm{B_t - B_{kh}} - \E\norm{B_t - B_{kh}})}^2}{8 \, (t-kh)}
    &\le 2\,,
\end{align*}
c.f.~\cite[Section 2.3, Theorem 5.5]{boucheronlugosimassart2013concentration}. Before proceeding, we state a very useful equality.
\begin{align}\label{eq:magic}
    \E_{\mu_t}\bigl[\psi_t \, \bigl\lVert \nabla \ln\bigl(\psi_t \, \frac{\D \mu_t}{\D \pi}\bigr) \bigr\rVert^2 \bigr]
    &= \frac{4\E_\pi[\norm{\nabla(\rho_t^{q/2})}^2]}{\E_\pi(\rho_t^q)}\,.
\end{align}
This will also be used in the subsequent proofs.
Next, using the LSI for $\pi$, we compute
\begin{align}
    \msf{KL}(\widetilde{\Pr} \mmid \Pr)
    &= \E_{\psi_t \mu_t}\ln \psi_t
    = \E_{\psi_t \mu_t} \ln \frac{\rho_t^{q-1}}{\E_{\mu_t}(\rho_t^{q-1})}
    = \frac{q-1}{q} \E_{\psi_t \mu_t} \ln \frac{\rho_t^q}{{\E_{\mu_t}(\rho_t^{q-1})}^{q/(q-1)}} \label{eq:surprising_calc_1} \\
    &= \frac{q-1}{q} \, \bigl\{ \E_{\psi_t \mu_t} \ln \frac{\rho_t^q}{\E_{\mu_t}(\rho_t^{q-1})} - \underbrace{\frac{1}{q-1} \ln \E_{\mu_t}(\rho_t^{q-1})}_{\ge 0}\bigr\}
    \le \frac{q-1}{q} \, \msf{KL}(\psi_t \mu_t \mmid \pi) \label{eq:surprising_calc_2} \\
    &\le \frac{(q-1)\, C_{\msf{LSI}}}{2q} \E_{\psi_t \mu_t}\bigl[\bigl\lVert\nabla \ln(\psi_t \, \frac{\D \mu_t}{\D \pi}\bigr)\bigr\rVert^2\bigr]
    = \frac{2 \, (q-1)\, C_{\msf{LSI}}}{q} \, \frac{\E_\pi[\norm{\nabla(\rho_t^{q/2})}^2]}{\E_\pi(\rho_t^q)}\,, \label{eq:surprising_calc_3}
\end{align}
where the last equality is~\eqref{eq:magic}.
We have thus proved
\begin{align*}
    &\E[\psi_t(x_t) \, \norm{B_t - B_{kh}}^2] \\
    &\qquad \le 2d \, (t-kh) + \frac{32h \, (q-1)\, C_{\msf{LSI}}}{q} \, \frac{\E_\pi[\norm{\nabla(\rho_t^{q/2})}^2]}{\E_\pi(\rho_t^q)} + (16\ln 2) \, (t-kh) \\
    &\qquad \le 14d \, (t-kh) + 32hC_{\msf{LSI}} \, \frac{\E_\pi[\norm{\nabla(\rho_t^{q/2})}^2]}{\E_\pi(\rho_t^q)}\,.
\end{align*}
This concludes the proof.
\end{proof}

Next, we formulate a lemma to control the expectation of $\norm{\nabla V}^2$ under a change of measure. Although this is not strictly necessary for the proof, it streamlines the argument.
\begin{lemma}\label{lem:fisher_info_arg}
    Assume that $\nabla V$ is $L$-Lipschitz.
    For any probability measure $\mu$, we have
    \begin{align*}
        \E_\mu[\norm{\nabla V}^2]
        &\le 4\E_\pi\bigl[\bigl\lVert \nabla \sqrt{\frac{\D\mu}{\D\pi}} \bigr\rVert^2\bigr] + 2dL
        = \E_\mu\bigl[\bigl\lVert \nabla \ln \frac{\D\mu}{\D\pi} \bigr\rVert^2\bigr] + 2dL\,.
    \end{align*}
\end{lemma}

\begin{proof}
    Let $\ms L$ denote the infinitesimal generator of the Langevin diffusion~\eqref{eq:langevin}, i.e., $\ms L f = \langle \nabla V, \nabla f \rangle - \Delta f$.
    Observe that $\ms L V = \norm{\nabla V}^2 - \Delta V$.
    Applying integration by parts and recalling that $\E_\pi \ms L f = 0$ for any $f$,
    \begin{align*}
        \E_\mu[\norm{\nabla V}^2]
        &= \E_\mu \ms L V + \E_\mu \Delta V
        \le \int \ms L V \, \bigl( \frac{\D\mu}{\D\pi} - 1 \bigr) \, \D \pi + dL
        = \int \bigl\langle \nabla V, \nabla \frac{\D\mu}{\D\pi} \bigr\rangle  \, \D \pi + dL \\
        &= 2 \int \bigl\langle \sqrt{\frac{\D\mu}{\D\pi}} \, \nabla V, \nabla \sqrt{\frac{\D\mu}{\D\pi}} \bigr\rangle \, \D \pi + dL \\
        &\le \frac{1}{2} \E_\mu[\norm{\nabla V}^2] + 2 \E_\pi\bigl[\bigl\lVert \nabla \sqrt{\frac{\D\mu}{\D\pi}} \bigr\rVert^2 \bigr] + dL\,.
    \end{align*}
    Rearrange this inequality to obtain the desired result.
\end{proof}

We are now ready to give the proof of Theorem~\ref{thm:lsi_result}. In order to emphasize the main ideas, we first present a proof which incurs a suboptimal dependence on $q$ and explain how to sharpen the argument afterwards.
\medskip

\begin{proof}[Proof of Theorem~\ref{thm:lsi_result}]
As encapsulated in the differential inequality of Proposition~\ref{prop:renyi_time_deriv_interpolation}, the crux of the proof of Theorem~\ref{thm:lsi_result} is to control the discretization error term $\E[\psi_t(x_t) \, \norm{\nabla V(x_t) - \nabla V(x_{kh})}^2]$ for $t \in [kh, (k+1)h]$.
Since $\nabla V$ is $L$-Lipschitz, we have $\norm{\nabla V(x_t) - \nabla V(x_{kh})}^2 \le 2L^2 \, {(t-kh)}^2 \, \norm{\nabla V(x_{kh})}^2 + 4L^2 \, \norm{B_t - B_{kh}}^2$. However, it is more convenient to have a bound in terms of $\norm{\nabla V(x_t)}$ rather than $\norm{\nabla V(x_{kh})}$, so we use
\begin{align*}
    \norm{\nabla V(x_{kh})}
    &\le \norm{\nabla V(x_t)} + L \, \norm{x_t - x_{kh}} \\
    &\le \norm{\nabla V(x_t)} + hL \, \norm{\nabla V(x_{kh})} + \sqrt 2 L \, \norm{B_t - B_{kh}}\,.
\end{align*}
If $h \le 1/(3L)$, we can rearrange this inequality to obtain $\norm{\nabla V(x_{kh})} \le \frac{3}{2} \, \norm{\nabla V(x_t)} + \frac{3L}{\sqrt 2} \, \norm{B_t - B_{kh}}$, so
\begin{align*}
    \norm{\nabla V(x_t) - \nabla V(x_{kh})}^2
    &\le 9L^2 \, {(t-kh)}^2 \, \norm{\nabla V(x_t)}^2 + (18h^2 L^4 +4L^2) \, \norm{B_t - B_{kh}}^2 \\
    &\le 9L^2 \, {(t-kh)}^2 \, \norm{\nabla V(x_t)}^2 + 6L^2 \, \norm{B_t - B_{kh}}^2\,.
\end{align*}
We will control the two error terms in turn.

For the first error term, applying Lemma~\ref{lem:fisher_info_arg} to the measure $\psi_t \mu_t$ and using \eqref{eq:magic} yields
\begin{align*}
    \E_{\psi_t \mu_t}[\norm{\nabla V}^2]
    &\le \E_{\mu_t}\bigl[\psi_t \, \bigl\lVert \nabla \ln\bigl(\psi_t \, \frac{\D \mu_t}{\D \pi}\bigr) \bigr\rVert^2 \bigr] + 2dL
    = \frac{\E_\pi[\rho_t^q \, \norm{\nabla \ln(\rho_t^q)}^2]}{\E_\pi(\rho_t^q)} + 2dL \\
    &= \frac{4\E_\pi[\norm{\nabla(\rho_t^{q/2})}^2]}{\E_\pi(\rho_t^q)} + 2dL\,.
\end{align*}
For the second error term, using Lemma \ref{lem:brownian_variational_arg}, we find as a result
\begin{align*}
    \E[\psi_t(x_t) \, \norm{B_t - B_{kh}}^2]
    &\le 14d \, (t-kh) + 32hC_{\msf{LSI}} \, \frac{\E_\pi[\norm{\nabla(\rho_t^{q/2})}^2]}{\E_\pi(\rho_t^q)}\,.
\end{align*}
Finally, collecting together the error terms and applying Proposition~\ref{prop:renyi_time_deriv_interpolation}, we see that
\begin{align*}
    \partial_t \eu R_q(\mu_t \mmid \pi)
    &\le -\frac{3}{q} \, \frac{\E_\pi[\norm{\nabla(\rho_t^{q/2})}^2]}{\E_\pi(\rho_t^q)} + 9qL^2 \, {(t-kh)}^2 \, \Bigl\{\frac{4\E_\pi[\norm{\nabla(\rho_t^{q/2})}^2]}{\E_\pi(\rho_t^q)} + 2dL\Bigr\} \\
    &\qquad\qquad\qquad\qquad\qquad {} + 6qL^2 \, \Bigl\{14d \, (t-kh) + 32hC_{\msf{LSI}} \, \frac{\E_\pi[\norm{\nabla(\rho_t^{q/2})}^2]}{\E_\pi(\rho_t^q)}\Bigr\}\,.
\end{align*}
Assuming for simplicity that $C_{\msf{LSI}}, L \ge 1$, then $h \le 1/(192q^2 C_{\msf{LSI}} L^2)$ implies
\begin{align*}
    \partial_t \eu R_q(\mu_t \mmid \pi)
    &\le -\frac{1}{q} \, \frac{\E_\pi[\norm{\nabla(\rho_t^{q/2})}^2]}{\E_\pi(\rho_t^q)} + 18dqL^3 \, {(t-kh)}^2 + 84dqL^2 \, (t-kh) \\
    &\le -\frac{1}{2qC_{\msf{LSI}}} \, \eu R_q(\mu_t \mmid \pi) + 18dqL^3 \, {(t-kh)}^2 + 84dqL^2 \, (t-kh)\,,
\end{align*}
where the last line uses the fact that $\pi$ satisfies LSI~\cite[see][Lemma 5]{vempalawibisono2019ula}. This then implies the differential inequality
\begin{align*}
    \partial_t\bigl\{\exp\bigl(\frac{t-kh}{2qC_{\msf{LSI}}}\bigr) \, \eu R_q(\mu_t \mmid \pi)\bigr\}
    &\le \exp\bigl(\frac{t-kh}{2qC_{\msf{LSI}}}\bigr) \, \{18dqL^3 \, {(t-kh)}^2 + 84dqL^2 \, (t-kh)\bigr\} \\
    &\le 19dqL^3 \, {(t-kh)}^2 + 85dqL^2 \, (t-kh)\,.
\end{align*}
Integrating this inequality over $t \in [kh, (k+1)h]$ yields the recursion
\begin{align*}
    \eu R_q(\mu_{(k+1)h} \mmid \pi)
    &\le \exp\bigl( - \frac{h}{2qC_{\msf{LSI}}}\bigr) \, \eu R_q(\mu_{kh} \mmid \pi) + \frac{19}{3} \, dh^3 qL^3 + \frac{85}{2} \, dh^2 qL^2 \\
    &\le \exp\bigl( - \frac{h}{2qC_{\msf{LSI}}}\bigr) \, \eu R_q(\mu_{kh} \mmid \pi) + 43dh^2 qL^2\,.
\end{align*}
Iterating this yields
    \begin{align*}
        \eu R_q(\mu_{Nh} \mmid \pi)
        &\le \exp\bigl( - \frac{Nh}{2qC_{\msf{LSI}}} \bigr) \, \eu R_q(\mu_0 \mmid \pi) + 86dhq^2 C_{\msf{LSI}} L^2 ,
    \end{align*}
    which completes the proof.
\end{proof}

We now outline the hypercontractivity argument to improve the dependence on $q$.

\begin{proposition}[Hypercontractivity] \label{prop:hypercontractivity}
    Let ${(\mu_t)}_{t\ge 0}$ denote the law of the interpolation~\eqref{eq:interpolated_lmc} of LMC\@.
    Also, let $q(t) \deq 1 + (q_0 - 1) \exp\frac{t}{2C_{\msf{LSI}}}$ for $t\ge 0$, and write $\rho_t \deq \frac{\D\mu_t}{\D\pi}$, $\psi_t \deq \rho_t^{q(t)-1}/\E_\pi(\rho_t^{q(t)})$. Then, for $t \in [kh, (k+1)h]$,
    \begin{align*}
        \partial_t \Bigl(\frac{1}{q(t)} \ln \int \rho_t^{q(t)} \, \D \pi\Bigr)
        &\le -\frac{2 \, (q(t)-1)}{{q(t)}^2} \,\frac{\E_\pi[\norm{\nabla(\rho_t^{q(t)/2})}^2]}{\E_\pi(\rho_t^{q(t)})} \\
        &\qquad{} + (q(t)-1) \E[\psi_t(x_t) \, \norm{\nabla V(x_t) - \nabla V(x_{kh})}^2]\,.
    \end{align*}
\end{proposition}
\begin{proof}
Using calculus together with Proposition~\ref{prop:interpolated_continuity_eq}, we compute the derivative in time as in Proposition~\ref{prop:renyi_time_deriv_interpolation}, only now taking into account the additional time-dependent function $q$.
Since the calculation is very similar to Proposition~\ref{prop:renyi_time_deriv_interpolation}, we only record the final result:
\begin{align*}
    &\partial_t \Bigl(\frac{1}{q(t)} \ln \int \rho_t^{q(t)} \, \D \pi\Bigr)
    = - \frac{1}{\E_\pi(\rho_t^{q(t)})} \int \langle \nabla(\rho_t^{q(t)-1}), \nabla \ln \rho_t + \Gamma_t \rangle \, \D \mu_t + \frac{\dot q(t) \ent_\pi(\rho_t^{q(t)})}{{q(t)}^2 \E_\pi(\rho_t^{q(t)})} \\
    &\qquad \le -\frac{3 \, (q(t)-1)}{{q(t)}^2} \, \frac{\E_\pi[\norm{\nabla (\rho_t^{q(t)/2})}^2]}{\E_\pi(\rho_t^{q(t)})} + (q(t)-1) \E[\psi_t(x_t) \, \norm{\nabla V(x_t) - \nabla V(x_{kh})}^2] \\
    &\qquad\qquad\qquad{} + \frac{\dot q(t) \ent_\pi(\rho_t^{q(t)})}{{q(t)}^2 \E_\pi(\rho_t^{q(t)})}\,,
\end{align*}
where $\dot q$ is the derivative of $q$, we write $\Gamma_t \deq \E[\nabla V(x_{kh}) \mid x_t = \cdot] - \nabla V$, and the entropy functional is defined in Section~\ref{scn:continuous-time}.
Applying~\eqref{eq:lsi},
\begin{align*}
    \frac{\dot q(t) \ent_\pi(\rho_t^{q(t)})}{{q(t)}^2 \E_\pi(\rho_t^{q(t)})}
    &\le \frac{2\dot q(t) C_{\msf{LSI}} \E_\pi[\norm{\nabla (\rho_t^{q(t)/2})}^2]}{{q(t)}^2 \E_\pi(\rho_t^{q(t)})}
    = \frac{q(t)-1}{{q(t)}^2} \, \frac{\E_\pi[\norm{\nabla (\rho_t^{q(t)/2})}^2]}{\E_\pi(\rho_t^{q(t)})}
\end{align*}
where the last equality follows from our choice of $q$.
\end{proof}

\begin{proof}[Proof of Theorem~\ref{thm:lsi_result}]
    \textbf{Initial waiting phase}.
    Let $\bar q \ge 3$.
    We apply Proposition~\ref{prop:hypercontractivity} with $q_0 = 2$ and for $t \le N_0 h$, where $N_0 = \lceil \frac{2C_{\msf{LSI}}}{h} \ln(\bar q - 1)\rceil$.
    As in the earlier proof of Theorem~\ref{thm:lsi_result}, we take $h \le 1/(192q^2 C_{\msf{LSI}} L^2)$; note that, $\bar q \le q(N_0 h) \le 2\bar q$. Then, the bound on the error term from the previous proof implies
    \begin{align*}
        \partial_t \Bigl(\frac{1}{q(t)} \ln \int \rho_t^{q(t)} \, \D \pi\Bigr)
        &\le 18d q(t) L^3 \, {(t-kh)}^2 + 84dq(t) L^2 \, (t-kh)\,.
    \end{align*}
    Integrating this over $t\in [kh, (k+1)h]$ yields
    \begin{align*}
        \frac{1}{q((k+1)h)} \ln\int \rho_{(k+1)h}^{q((k+1)h)} \, \D \pi
        - \frac{1}{q(kh)} \ln\int \rho_{kh}^{q(kh)} \, \D \pi
        &\le 12dh^3 \bar q L^3 + 84dh^2\bar q L^2 \\
        &\le 85dh^2\bar q L^2\,.
    \end{align*}
    Iterating this yields
    \begin{align*}
        \frac{1}{q(N_0 h)} \ln \int \rho_{N_0 h}^{q(N_0 h)} \, \D \pi
        - \frac{1}{2}\ln \int \rho_0^2 \, \D \pi
        \le 85dh^2\bar q L^2 N_0
        \le 170dh\bar q C_{\msf{LSI}} L^2 \ln \bar q\,.
    \end{align*}
    
    \textbf{Remainder of the convergence analysis}. After shifting the time indices and applying the preceding proof of Theorem~\ref{thm:lsi_result} with $q = 2$,
    \begin{align*}
        \eu R_{\bar q}(\mu_{(N+N_0) h} \mmid \pi)
        &\le \frac{3}{2\bar q} \ln \int \rho_{(N+N_0) h}^{\bar q} \, \D \pi
        \le \frac{3}{4} \, \eu R_2(\mu_{Nh} \mmid \pi) + 255dh\bar q C_{\msf{LSI}} L^2 \ln \bar q \\
        &\le \frac{3}{4} \exp\bigl(-\frac{Nh}{4C_{\msf{LSI}}}\bigr) \, \eu R_2(\mu_0 \mmid \pi) + 258dhC_{\msf{LSI}} L^2 + 255dh\bar q C_{\msf{LSI}} L^2 \ln \bar q \\
        &\le \exp\bigl(-\frac{Nh}{4C_{\msf{LSI}}}\bigr) \, \eu R_2(\mu_0 \mmid \pi) + 513dh\bar q C_{\msf{LSI}} L^2 \ln \bar q\,.
    \end{align*}
    This completes the proof.
\end{proof}

\subsection{Proof of Theorem~\ref{thm:log_concave_result}}\label{scn:proof_log_concave_result}

The proof of Theorem~\ref{thm:log_concave_result} builds upon the proof of Theorem~\ref{thm:lsi_result}.
\medskip

\begin{lemma}\label{lem:lsi_along_iterates_cvx}
    Assume that $V$ is convex and $\nabla V$ is $L$-Lipschitz. Let ${(\mu_{kh})}_{k\in\N}$ denote the law of the iterates of \ref{eq:lmc} initialized at $\mu_0 = \normal(0, L^{-1} I_d)$ and run with step size $h \le 1/L$. 
    Then, the LSI constant $C_{\msf{LSI}}(\mu_{kh})$ of $\mu_{kh}$ satisfies $C_{\msf{LSI}}(\mu_{kh}) \le L + 2kh$.
\end{lemma}
\begin{proof}
    With the condition on the step size, ${\id} - h\nabla V$ is a contraction. Using standard facts about the behavior of the log-Sobolev constant under contractions (\cite[Proposition 5.4.3]{bakrygentilledoux2014}) and convolutions (see e.g.~\cite[Corollary 3.1]{chafai2004phientropies}), we obtain
    \begin{align*}
        C_{\msf{LSI}}(\mu_{(k+1)h})
        &\le C_{\msf{LSI}}\bigl({({\id} - h\nabla V)}_\# \mu_{kh}\bigr) + 2h \le C_{\msf{LSI}}(\mu_{kh}) + 2h\,.
    \end{align*}
    The result follows via iteration.
\end{proof}

\begin{proof}[Proof of Theorem~\ref{thm:log_concave_result}]
Using again the differential inequality of Proposition~\ref{prop:renyi_time_deriv_interpolation}, assuming $h \le 1/(3L)$, we want to control the error term
\begin{align*}
    &\E[\psi_t(x_t) \, \norm{\nabla V(x_t) - \nabla V(x_{kh})}^2] \\
    &\qquad \le 9L^2 \, {(t-kh)}^2 \, \E[\psi_t(x_t) \, \norm{\nabla V(x_t)}^2] + 6L^2 \, \E[\psi_t(x_t) \, \norm{B_t - B_{kh}}^2]\,,
\end{align*}
see the first proof of Theorem~\ref{thm:lsi_result}. For the first term, an application of Lemma~\ref{lem:fisher_info_arg} again yields
\begin{align*}
    \E[\psi_t(x_t) \, \norm{\nabla V(x_t)}^2]
    &\le \frac{4\E_\pi[\norm{\nabla (\rho_t^{q/2})}^2]}{\E_\pi(\rho_t^q)} + 2dL\,.
\end{align*}

For the second term, the Donsker--Varadhan variational principle~\eqref{eq:brownian_donsker_varadhan} implies
\begin{align*}
    \E[\psi_t(x_t) \, \norm{B_t - B_{kh}}^2]
    &\le 2d \, (t-kh) + 16 \, (t-kh)\, \{\msf{KL}(\widetilde\Pr \mmid \Pr) + \ln 2\}\,.
\end{align*}
Now comes a key difference in the proof: in Theorem~\ref{thm:lsi_result}, we bounded $\msf{KL}(\widetilde\Pr \mmid \Pr) \le \frac{q-1}{q} \, \msf{KL}(\psi_t \mu_t \mmid \pi)$ and applied the LSI for $\pi$.
Here, we instead use $\msf{KL}(\widetilde\Pr \mmid \Pr) = \msf{KL}(\psi_t\mu_t \mmid \mu_t)$ and apply the LSI from Lemma~\ref{lem:lsi_along_iterates_cvx} which worsens over time. We thus obtain
\begin{align*}
    \msf{KL}(\widetilde\Pr \mmid \Pr)
    &\le 2C_{\msf{LSI}}(\mu_t) \, \frac{\E_\pi[\norm{\nabla(\rho_t^{q/2})}^2]}{\E_\pi(\rho_t^q)}
    \le 2 \, (L + 2 \, (k+1)\, h) \, \frac{\E_\pi[\norm{\nabla(\rho_t^{q/2})}^2]}{\E_\pi(\rho_t^q)}\,.
\end{align*}

Let $N$ denote the total number of iterations that we run \ref{eq:lmc}.
Collecting together all of the error terms and using Proposition~\ref{prop:renyi_time_deriv_interpolation}, we see that
\begin{align*}
    \partial_t \eu R_q(\mu_t \mmid \pi)
    &\le -\frac{3}{q} \, \frac{\E_\pi[\norm{\nabla(\rho_t^{q/2})}^2]}{\E_\pi(\rho_t^q)} + 9qL^2 \, {(t-kh)}^2 \, \Bigl\{\frac{4\E_\pi[\norm{\nabla(\rho_t^{q/2})}^2]}{\E_\pi(\rho_t^q)} + 2dL\Bigr\} \\
    &\qquad\qquad\qquad{} + 6qL^2 \, \Bigl\{14d \, (t-kh) + 32h \, (L+2Nh) \, \frac{\E_\pi[\norm{\nabla(\rho_t^{q/2})}^2]}{\E_\pi(\rho_t^q)}\Bigr\}\,.
\end{align*}
Assuming $h \le \frac{1}{384qL\sqrt N} \min\{1, \frac{\sqrt N}{qL^2}\}$, it yields
\begin{align*}
    \partial_t \eu R_q(\mu_t \mmid \pi)
    &\le -\frac{1}{q} \, \frac{\E_\pi[\norm{\nabla(\rho_t^{q/2})}^2]}{\E_\pi(\rho_t^q)} + 18dqL^3 \, {(t-kh)}^2 + 84dqL^2 \, (t-kh) \\
    &\le -\frac{1}{qC_{\msf{PI}}} \, \{1-\exp(-\eu R_q(\mu_t \mmid \pi))\} + 18dqL^3 \, {(t-kh)}^2 + 84dqL^2 \, (t-kh)\,,
\end{align*}
where the last inequality follows from~\cite[Lemma 17]{vempalawibisono2019ula}. 

We now split the analysis into two phases.
In the first phase, we consider $t\le N_0 h$, where $N_0$ is the largest integer such that $\eu R_q(\mu_{N_0 h} \mmid \pi) \ge 1$. Then,
\begin{align*}
    \partial_t \eu R_q(\mu_t \mmid \pi)
    &\le -\frac{1}{2qC_{\msf{PI}}} + 18dqL^3 \, {(t-kh)}^2 + 84dqL^2 \, (t-kh)\,.
\end{align*}
Integration yields
\begin{align*}
    \eu R_q(\mu_{(k+1)h} \mmid \pi) - \eu R_q(\mu_{kh}\mmid \pi)
    &\le -\frac{h}{2qC_{\msf{PI}}} + 6dh^3 qL^3 + 42dh^2 qL^2 \\
    &\le -\frac{h}{2qC_{\msf{PI}}} + 43dh^2 qL^2\,.
\end{align*}
If $h \le \frac{1}{172dq^2 C_{\msf{PI}} L^2}$, then we deduce that $\eu R_q(\mu_{kh} \mmid \pi) \le \eu R_q(\mu_0 \mmid \pi) - \frac{kh}{4qC_{\msf{PI}}}$, and hence that the first phase ends after at most $N_0  \le 4q C_{\msf{PI}}\eu R_q(\mu_0\mmid \pi)/h$ iterations.

In the second phase, we consider $t$ such that $\eu R_q(\mu_t \mmid \pi) \le 1$. Using $1-\exp(-x) \ge x/2$ for $x \in [0, 1]$, in this phase we have the inequality
\begin{align*}
    \partial_t \eu R_q(\mu_t \mmid \pi)
    &\le -\frac{1}{2qC_{\msf{PI}}} \, \eu R_q(\mu_t \mmid \pi) + 18dqL^3 \, {(t-kh)}^2 + 84dqL^2 \, (t-kh)\,.
\end{align*}
As in the proof of Theorem~\ref{thm:lsi_result}, it implies
\begin{align*}
    \eu R_q(\mu_{Nh} \mmid \pi)
    &\le \exp\bigl( - \frac{(N - N_0 - 1) h}{2qC_{\msf{PI}}}\bigr) \, \eu R_q(\mu_{(N_0+1)h} \mmid \pi) + 88dhq^2 C_{\msf{PI}} L^2 \\
    &\le \exp\bigl( - \frac{(N - N_0 - 1) h}{2qC_{\msf{PI}}}\bigr) + 88dhq^2 C_{\msf{PI}} L^2\,.
\end{align*}
To make this at most $\varepsilon$, we take $h \le \frac{\varepsilon}{176dq^2 C_{\msf{PI}} L^2}$ and $N \ge N_0 + 1 + \frac{2qC_{\msf{PI}}}{h} \ln(2/\varepsilon)$.

From Lemma~\ref{lem:initialization_cvx}, we see that $\eu R_q(\mu_0 \mmid \pi) = \widetilde O(d)$, so that $N = \widetilde \Theta(\frac{dq C_{\msf{PI}}}{h})$. Substituting this into our earlier constraints on $h$, we see that if we take
\begin{align*}
    h
    &= \widetilde \Theta\Bigl( \frac{\varepsilon}{dq^2 C_{\msf{PI}} L^2} \min\bigl\{1, \frac{1}{q\varepsilon}, \frac{dC_{\msf{PI}}}{\varepsilon L}\bigr\}\Bigr)\,,
\end{align*}
then the iteration complexity is
\begin{align*}
    N
    &= \widetilde \Theta\Bigl( \frac{d^2 q^3 C_{\msf{PI}}^2 L^2}{\varepsilon} \max\bigl\{1, q\varepsilon, \frac{\varepsilon L}{dC_{\msf{PI}}}\bigr\}\Bigr)\,.
\end{align*}
This completes the proof.
\end{proof}

\subsection{Proof of Theorem~\ref{thm:main}}\label{scn:proof_main}

\subsubsection{Girsanov's theorem and change of measure}

As discussed in Section~\ref{scn:discretization_err}, we will use the Girsanov's theorem, stated below in a form which is convenient for our purposes.

\begin{theorem}[{Girsanov's theorem,~\cite[Theorem 8.6.8]{oksendal2013stochastic}}]\label{thm:girsanov}
    Let ${(x_t)}_{t\ge 0}$, ${(b_t^P)}_{t\ge 0}$, ${(b_t^Q)}_{t\ge 0}$ be stochastic processes adapted to the same filtration. Let $P_T$ and $Q_T$ be probability measures on the path space $C([0,T]; \R^d)$ and ${(x_t)}_{t\ge 0}$ evolves according to
    \begin{align*}
        \D x_t
        &= b_t^P \, \D t + \sqrt 2 \, \D B_t^P \qquad\text{under}~P_T\,, \\
        \D x_t
        &= b_t^Q \, \D t + \sqrt 2 \, \D B_t^Q \qquad\text{under}~Q_T\,,
    \end{align*}
    where $B^P$ is a $P_T$-Brownian motion and $B^Q$ is a $Q_T$-Brownian motion.
    Assume that Novikov's condition,
    \begin{align*}
        \E^{Q_T} \exp\Bigl(\frac{1}{4}\int_0^T \norm{b_t^P - b_t^Q}^2 \, \D t\Bigr) < \infty,
    \end{align*}
    holds.
    Then,
    \begin{align*}
        \frac{\D P_T}{\D Q_T}
        &= \exp\Bigl(\frac{1}{\sqrt 2} \int_0^T \langle b_t^P - b_t^Q, \D B_t^Q \rangle - \frac{1}{4} \int_0^T \norm{b_t^P - b_t^Q}^2 \, \D t\Bigr)\,.
    \end{align*}
\end{theorem}

\begin{remark}
In our applications of Girsanov's theorem, although we do not check Novikov's condition explicitly, the validity of Novikov's condition follows from the proof. In particular, it is a smaller quantity than the one we control in inequality~\eqref{eq:stronger_than_novikov}, and it can be bounded by exactly following the steps in the proof of Theorem~\ref{thm:main_disc_bd}.
\end{remark}

Actually, we only need the following corollary.

\begin{corollary}\label{cor:girsanov}
    For any $q \ge 1$,
    \begin{align*}
        \E^{Q_T}\bigl[\bigl( \frac{\D P_T}{\D Q_T} \bigr)^q\bigr]
        &\le \sqrt{\E\exp\Bigl( q^2 \int_0^T \norm{b_t^P - b_t^Q}^2 \, \D t\Bigr)}\,.
    \end{align*}
\end{corollary}

\begin{proof}
    Applying the Cauchy--Schwarz inequality,
    \begin{align*}
        \E^{Q_T}\bigl[\bigl( \frac{\D P_T}{\D Q_T} \bigr)^q \bigr]
        &= \E^{Q_T}\Bigl[ \exp\Bigl(\frac{q}{\sqrt 2} \int_0^T \langle b_t^P - b_t^Q, \D B_t^Q \rangle - \frac{q}{4} \int_0^T \norm{b_t^P - b_t^Q}^2 \, \D t\Bigr) \Bigr] \\
        &\le \sqrt{\E^{Q_T}\Bigl[\exp\Bigl(\bigl(q^2 - \frac{q}{2}\bigr) \int_0^T \norm{b_t^P - b_t^Q}^2 \, \D t\Bigr) \Bigr]} \\
        &\qquad{} \times\underbrace{\sqrt{\E^{Q_T} \exp\Bigl(\sqrt 2 q \int_0^T \langle b_t^P - b_t^Q, \D B_t^Q \rangle - q^2 \int_0^T \norm{b_t^P - b_t^Q}^2 \, \D t\Bigr)}}_{=1} \\
        &\le \sqrt{\E^{Q_T}\Bigl[\exp\Bigl(q^2 \int_0^T \norm{b_t^P - b_t^Q}^2 \, \D t\Bigr)\Bigr]}\,,
    \end{align*}
    where we used It\^o's lemma to show that the underlined term equals $1$.
\end{proof}

Next, we state and prove the change of measure principle described in Section~\ref{scn:discretization_overview}. This lemma will be invoked repeatedly in the main arguments.

\begin{lemma}[change of measure]\label{lem:change_of_measure}
    Let $\mu$, $\nu$ be probability measures and let $E$ be any event.
    Then,
    \begin{align*}
        \mu(E)
        &\le \nu(E) + \sqrt{\chi^2(\mu \mmid \nu) \, \nu(E)}\,.
    \end{align*}
    In particular, if $\mu$ and $\nu$ are probability measures on $\R^d$ and
    \begin{align*}
        \nu\{\norm \cdot \ge R_0 + \eta\} \le C\exp(-c\eta^2) \qquad\text{for all}~\eta \ge 0\,,
    \end{align*}
    where $C \ge 1$, then
    \begin{align*}
        \mu\Bigl\{\norm \cdot \ge R_0 + \sqrt{\frac{1}{c} \, \eu R_2(\mu\mmid \nu)} + \eta\Bigr\} \le 2C\exp\bigl(-\frac{c\eta^2}{2}\bigr) \qquad\text{for all}~\eta \ge 0\,.
    \end{align*}
\end{lemma}
\begin{proof}
    \begin{align*}
        \mu(E)
        &= \nu(E) + \int \one_E \, \bigl(\frac{\D\mu}{\D\nu} - 1\bigr) \,\D \nu
        \le \nu(E) + \sqrt{\chi^2(\mu \mmid \nu) \, \nu(E)}\,,
    \end{align*}
    where the last inequality is the Cauchy--Schwarz inequality.
    For the second statement, applying the change of measure principle to $E = \{\norm \cdot \ge R_0 + \bar\eta\}$ yields
    \begin{align*}
        \mu\{\norm\cdot \ge R_0 + \bar\eta\}
        &\le C\exp( - c\bar\eta^2 ) + \sqrt{C\exp\bigl\{-\bigl( c\bar\eta^2 - \eu R_2(\mu \mmid \nu)\bigr)\bigr\}}\,.
    \end{align*}
    Now take $\bar\eta = \sqrt{\frac{1}{c} \, \eu R_2(\mu \mmid \nu)} + \eta$.
\end{proof}

\subsubsection{Sub-Gaussianity of the Langevin diffusion}

In this section, we introduce a modified distribution: for $\gamma, R > 0$,
\begin{align}\label{eq:modified_potential}
    \hat \pi \propto \exp(-\hat V)\,, \qquad \hat V(x) \deq V(x) + \frac{\gamma}{2} \, {(\norm x - R)}_+^2\,.
\end{align}
Here, ${(\norm x - R)}_+^2$ is interpreted as ${\max\{\norm x - R, 0\}}^2$.
Although $\hat\pi$ and $\hat V$ depend on the parameters $\gamma$ and $R$, we will suppress this in the notation for simplicity. Note that by construction, $V = \hat V$ on the ball $B(0,R)$ of radius $R$ centered at the origin. Also, the probability measure $\hat\pi$ has sub-Gaussian tails.
We record this and other useful facts below.

\begin{lemma}[properties of the modified potential]\label{lem:modified_potential}
    Let $\hat\pi$ and $\hat V$ be defined as in~\eqref{eq:modified_potential}. Assume that $\nabla V(0) = 0$ and that $\nabla V$ satisfies~\eqref{eq:holder}.
    Then, the following assertions hold.
    \begin{enumerate}
        \item (sub-Gaussian tail bound) Assume that $R$ is chosen so that $\pi(B(0,R)) \ge 1/2$.
         Then, for all $\eta \ge 0$,
         \begin{align*}
             \hat\pi\{\norm \cdot \ge R + \eta\}
            \le 2\exp\bigl(-\frac{\gamma \eta^2}{2}\bigr)\,.
         \end{align*}
         \item (gradient growth) The gradient $\nabla \hat V$ satisfies
         \begin{align*}
             \norm{\nabla \hat V(x)}
             &\le L + (L+\gamma) \, \norm x\,.
         \end{align*}
    \end{enumerate}
\end{lemma}
\begin{proof}
    \begin{enumerate}
        \item We can write
\begin{align*}
    \int \exp\bigl(\frac{\gamma}{2} \, {(\norm \cdot - R)}_+^2\bigr) \, \D \hat\pi
    &= \frac{\int \exp(-V)}{\int \exp(-\hat V)}\,.
\end{align*}
Next, we bound
\begin{align*}
    \frac{\int \exp(-\hat V)}{\int \exp(-V)}
    = \int \exp\bigl(-\frac{\gamma}{2} \, {(\norm \cdot - R)}_+^2\bigr) \, \D \pi
    \ge \pi\bigl(B(0,R)\bigr)
    \ge \frac{1}{2}
\end{align*}
by our assumption on $R$.
The sub-Gaussian tail bound follows from Markov's inequality via
\begin{align*}
    \hat\pi\{\norm \cdot - R \ge \eta\}
    &\le \hat\pi\Bigl\{\exp\bigl(\frac{\gamma}{2} \, {(\norm \cdot - R)}_+^2\bigr) \ge \exp \frac{\gamma \eta^2}{2}\Bigr\}
    \le 2\exp\bigl(-\frac{\gamma \eta^2}{2}\bigr)\,.
\end{align*}
\item First, note that $\norm{\nabla V(x)} \le L \, \norm x^s \le L \, (1+\norm x)$, using $\nabla V(0) = 0$ and~\eqref{eq:holder}. Then,
\begin{align*}
    \norm{\nabla \hat V(x)}
    &\le \norm{\nabla V(x)} + \gamma \, {(\norm x - R)}_+
    \le L + (L+\gamma) \, \norm x\,.
\end{align*}
    \end{enumerate}
\end{proof}

Throughout this section, we will assume that $R \ge \max\{1, 2\mf m\}$, where $\mf m \deq \int \norm \cdot \, \D \pi$, so that the sub-Gaussian tail bound in Lemma~\ref{lem:modified_potential} is valid.

We now begin transferring the sub-Gaussianity of $\hat\pi$ to $\pi_t$. First, we establish sub-Gaussian tail bounds for $\hat\pi_t$, where ${(\hat\pi_t)}_{t\ge 0}$ is the law of the continuous-time Langevin diffusion
\begin{align}\label{eq:modified_diffusion}
    \D\hat x_t
    &= -\nabla \hat V(\hat x_t) \, \D t + \sqrt 2 \, \D B_t
\end{align}
with potential $\hat V$, initialized at $\hat x_0 \sim \mu_0$.

\begin{lemma}\label{lem:modified_diffusion_sg}
    Let ${(\hat z_t)}_{t\ge 0}$ denote the modified diffusion~\eqref{eq:modified_diffusion} with potential $\hat V$.
    Assume that $h \le 1/(2 \, (L+\gamma))$ and $R\ge \max\{1,2\mf m\}$.
    Then, for all $\delta \in (0,1)$ and any $k = 0, 1, \ldots, N-1$, with probability at least $1-\delta$,
    \begin{align*}
        \sup_{t\in [kh, (k+1)h]}{\norm{\hat z_t}}
        &\le R + 4h \, (L+\gamma) \, R + \sqrt{\frac{8}{\gamma} \, \eu R_2(\mu_0 \mmid \hat \pi)} + \sqrt{\bigl(96dh + \frac{32}{\gamma}\bigr) \ln \frac{1}{\delta}}\,.
    \end{align*}
\end{lemma}
\begin{proof}
    Apply the change of measure principle (Lemma~\ref{lem:change_of_measure}) together with the sub-Gaussian tail bound in Lemma~\ref{lem:modified_potential} to see that with probability at least $1-\delta$ for any $k = 0, 1, \ldots, N-1$,
    \begin{align}\label{eq:modified_diffusion_sg_1}
        \norm{\hat z_{kh}}
        &\le R + \sqrt{\frac{2}{\gamma} \, \eu R_2(\hat \pi_{kh} \mmid \hat\pi)} + \sqrt{\frac{4}{\gamma} \ln \frac{4}{\delta}}\,.
    \end{align}
    Since the R\'enyi divergence is decreasing along the diffusion~\eqref{eq:modified_diffusion}, then $\eu R_2(\hat\pi_t \mmid \hat\pi) \le \eu R_2(\mu_0 \mmid \hat\pi)$.
    
    Next, for $t \leq h$,
\begin{align*}
    \norm{\hat z_{kh + t} - \hat z_{kh}}
    &\leq \int_0^t \norm{\nabla \hat V(\hat z_{kh+r})} \, \D r +  \sqrt{2}\, \norm{B_{kh+t} - B_{kh}} \\
    &\le hL + (L+\gamma) \int_0^t \norm{\hat z_{kh+r}} \, \D r + \sqrt 2 \, \norm{B_{kh+t} - B_{kh}} \\
    &\le hL + (L+\gamma) \, \Bigl(h \,\norm{\hat z_{kh}} + \int_0^t \norm{\hat z_{kh+r} - \hat z_{kh}} \, \D r\Bigr) + \sqrt 2 \, \norm{B_{kh+t} - B_{kh}}\,,
\end{align*}
where we used Lemma~\ref{lem:modified_potential}. Gr\"onwall's inequality implies
\begin{align*}
    \sup_{t\in [0,h]}{\norm{\hat z_{kh + t} - \hat z_{kh}}}
    &\le \bigl(hL + h \, (L+\gamma) \, \norm{\hat z_{kh}} + \sqrt 2 \sup_{t\in [0,h]}{\norm{B_{kh+t} - B_{kh}}}\bigr) \exp\bigl(h \, (L+\gamma)\bigr) \\
    &\le 2hL + 2h \, (L+\gamma) \, \norm{\hat z_{kh}} + \sqrt{8} \sup_{t\in [0,h]}{\norm{B_{kh+t} - B_{kh}}}
\end{align*}
provided $h \le 1/(2 \, (L+\gamma))$. Now, a union bound shows that
\begin{align*}
    &\Pr\bigl\{\sup_{t\in [kh, (k+1)h]}{\norm{\hat z_t}} \ge \eta\bigr\} \\
    &\qquad \le \Pr\bigl\{{\norm{\hat z_{kh}}} \ge R'\bigr\} \\
    &\qquad\qquad\qquad{} + \Pr\Bigl\{ \sup_{t\in [0,h]}{\norm{\hat z_{kh + t} - \hat z_{kh}}} \ge \eta - R', \; {\norm{\hat z_{kh}}} \le R'\Bigr\} \\
    &\qquad\le \Pr\bigl\{{\norm{\hat z_{kh}}} \ge R'\bigr\} \\
    &\qquad\qquad\qquad{} + \Pr\Bigl\{\sqrt 8 \sup_{t\in [0,h]}{\norm{B_{kh+t} - B_{kh}}} \ge \eta - R' - 2hL - 2h \, (L+\gamma) \, R'\Bigr\}\,.
\end{align*}
Taking $R' = R + \sqrt{\frac{2}{\gamma} \, \eu R_2(\mu_0 \mmid \hat\pi)} + \sqrt{\frac{4}{\gamma} \ln \frac{1}{\delta}}$ and applying a standard bound on the tail probability of Brownian motion (Lemma~\ref{lem:brownian_mgf}) shows that with probability at least $1-\delta$, if $R \ge 1$, 
\begin{align*}
    \sup_{t\in [kh, (k+1)h]}{\norm{\hat z_t}}
    &\le \eta = R' + 2hL + 2h \, (L+\gamma) \, R' + \sqrt{48dh \ln \frac{1}{\delta}} \\
    &\le R + 4h \, (L+\gamma) \, R + \sqrt{\frac{8}{\gamma} \, \eu R_2(\mu_0 \mmid \hat \pi)} + \sqrt{\bigl(96dh + \frac{32}{\gamma}\bigr) \ln \frac{1}{\delta}}
\end{align*}
after simplifying some terms.
\end{proof}

Next, we control the R\'enyi divergence between $\pi_t$ and $\hat\pi_t$, which ultimately allows us to transfer the sub-Gaussianity to $\pi_t$.

\begin{proposition}
    Let $T \deq Nh$.
    Let $Q_T$, $\hat Q_T$ be the measures on path space corresponding to the original diffusion~\eqref{eq:langevin} and the modified diffusion~\eqref{eq:modified_diffusion} respectively, both initialized at $\mu_0$. Assume that $h \le \frac{1}{3} \min\{\frac{1}{L+\gamma}, \frac{T}{d}\}$ and $\gamma \leq \frac{1}{768 T}$.
    Also, suppose that $R \ge \max\{1, 2\mf m\}$ and $\eu R_2(\mu_0 \mmid \hat \pi) \ge 1$. Then,
    \begin{align*}
        \eu R_2(Q_T \mmid \hat Q_T)
        &\le \frac{h \, {(L+\gamma)}^2 \, R^2}{d} + 5\eu R_2(\mu_0 \mmid \hat \pi)\,.
    \end{align*}
\end{proposition}
\begin{proof}
Applying Girsanov's theorem through Corollary~\ref{cor:girsanov} and the AM{--}GM inequality, we obtain
\begin{align*}
    \Bigl\{\E\bigl[ \bigl(\frac{\D Q_T}{\D \hat Q_T}\bigr)^2\bigr]\Bigr\}^2
    &\leq \E\exp\Bigl(4\int_0^T \norm{\nabla V(\hat z_t) - \nabla \hat V(\hat z_t)}^2 \, \D t\Bigr) \\
    &= \E\exp\Bigl(4\gamma^2 \sum_{k=0}^{N-1} \int_{kh}^{(k+1)h} (\norm{\hat z_t} - R)^2_+ \, \D t\Bigr)  \\
    &\leq \frac{1}{N} \sum_{k=0}^{N-1} \E \exp\Bigl(4\gamma^2 N \int_{kh}^{(k+1)h} (\norm{\hat z_t} - R)^2_+ \, \D t\Bigr)\,.
\end{align*}
Using the sub-Gaussian tail bound from Lemma~\ref{lem:modified_diffusion_sg}, we recall that with probability $1-\delta$ for $\delta \in (0, 1)$,
\begin{align*}
    \sup_{t\in [kh, (k+1)h]}{(\norm{\hat z_t} - R)^2_+}\le 12h^2 \, (L+\gamma)^2 \, R^2 + \frac{24}{\gamma} \, \eu R_2(\mu_0 \mmid \hat \pi) + \bigl(288dh + \frac{96}{\gamma}\bigr) \ln \frac{1}{\delta}\,.
\end{align*}
Consequently, integrating each of these tail bounds, it suffices to take $\gamma, h$ such that $1152\gamma^2 dhT + 384\gamma T < 1$
\begin{align*}
    \ln \E\bigl[ \bigl(\frac{\D Q_T}{\D \hat Q_T}\bigr)^2\bigr] \leq \bigl(24 \gamma^2  h^2\, (L+\gamma)^2\, R^2 T + 48\gamma T\, \eu R_2(\mu_0 \mmid \hat \pi) + (1152 \gamma^2 dh T+ 384\gamma T) \bigr)\,.
\end{align*}
From the condition above, it is sufficent to take $\gamma \leq \frac{1}{768T}$, $h \leq \frac{T}{d}$. Then, we immediately obtain the bound
\begin{align*}
    &\eu R_2(Q_T \mmid \hat Q_T)
    = \ln \E\bigl[ \bigl(\frac{\D Q_T}{\D \hat Q_T}\bigr)^2 \bigr] \\
    &\qquad \le \bigl(24 \gamma^2 h^2\, (L+\gamma)^2\, R^2 + 48\gamma\, \eu R_2(\mu_0 \mmid \hat \pi) + (1152 \gamma^2 dh+ 384\gamma)\bigr)\, T \\
    &\qquad \le \frac{h^2 \, {(L+\gamma)}^2 \, R^2}{T} + \eu R_2(\mu_0 \mmid \hat \pi) + \frac{dh}{T} + 1 \\
    &\qquad \le \frac{h \, {(L+\gamma)}^2 \, R^2}{d} + 5\eu R_2(\mu_0 \mmid \hat \pi)\,,
\end{align*}
where we have combined terms using $\eu R_2(\mu_0 \mmid \hat \pi) \ge 1$ to simplify the final bound.
\end{proof}

\begin{proposition}\label{prop:main_sg_bd}
    Let ${(z_t)}_{t\ge 0}$ denote the continuous-time diffusion~\eqref{eq:langevin} initialized at $\mu_0$.
    Assume that $h \le \frac{1}{3} \min\{\frac{1}{L+T^{-1}}, \frac{T}{d}\}$ and $\mf m, \eu R_2(\mu_0 \mmid \hat \pi) \ge 1$.
    Then, for all $\delta \in (0,1/2)$, with probability at least $1-\delta$ and for any fixed choice of $k = 0, 1, \ldots, N-1$
    \begin{align*}
        \norm{z_{kh}}
        &\le 2\mf m + 490 \sqrt{T\eu R_2(\mu_0 \mmid \hat \pi)}
        + \frac{230 h^{1/2}\mf m \, (L+T^{-1}) \, T^{1/2}}{d^{1/2}} + 160\sqrt{T\ln \frac{1}{\delta}}\,,
    \end{align*}
    where we write $T \deq Nh$.
\end{proposition}
\begin{proof}
    Recall from the proof of Lemma~\ref{lem:modified_diffusion_sg} that with probability at least $1-\delta$,
    \begin{align*}
        {\norm{\hat z_{kh}}}
        &\le R + \sqrt{\frac{2}{\gamma} \, \eu R_2(\mu_0 \mmid \hat\pi)} + \sqrt{\frac{4}{\gamma} \ln \frac{1}{\delta}}
    \end{align*}
    (see~\eqref{eq:modified_diffusion_sg_1}). Equivalently,
    \begin{align*}
        \Pr\Bigl\{{\norm{\hat z_{kh}}} \ge R + \sqrt{\frac{2}{\gamma} \, \eu R_2(\mu_0 \mmid \hat\pi)} + \eta\Bigr\} \le \exp\bigl( - \frac{\gamma \eta^2}{4}\bigr)\,.
    \end{align*}
    Applying the change of measure principle (Lemma~\ref{lem:change_of_measure}) again to $Q_T$ and $\hat Q_T$ with the choice $\gamma \leq \frac{1}{768T}$ and $R = 2\mf m$ reveals that for all $\delta \in (0, 1/2)$, with probability at least $1-\delta$,
    \begin{align*}
        \norm{z_{kh}}
        &\le R + \sqrt{\frac{2}{\gamma} \, \eu R_2(\mu_0 \mmid \hat\pi)} + \sqrt{\frac{4}{\gamma} \, \eu R_2(Q_T \mmid \hat Q_T)} + \sqrt{\frac{8}{\gamma} \ln \frac{1}{\delta}} \\
        &\le 2\mf m + 490 \sqrt{T\eu R_2(\mu_0 \mmid \hat \pi)} + \frac{230 h^{1/2}\mf m \, (L+T^{-1}) \, T^{1/2}}{d^{1/2}} + 160\sqrt{T\ln \frac{1}{\delta}}\,,
    \end{align*}
    after simplifying the bound.
\end{proof}

\subsubsection{Bounding the discretization error}

In this section, we prove our main bound on the discretization error.

\medskip{}

\begin{proof}[{Proof of Theorem~\ref{thm:main_disc_bd}}]
    Let $P$, $Q$ denote the measures on path space corresponding to the interpolated process~\eqref{eq:interpolated_lmc} and the continuous-time diffusion~\eqref{eq:langevin} respectively, both initialized at $\mu_0$.
    First, Corollary~\ref{cor:girsanov} gives
    \begin{align}
    \label{eq:stronger_than_novikov}
        \Bigl\{\E^{Q_T}\bigl[\bigl(\frac{\D P_T}{\D Q_T}\bigr)^q\bigr]\Bigr\}^2
        &\le \E^{Q_T} \exp\Bigl(q^2 \int_0^T \norm{\nabla V(z_t) - \nabla V(z_{\lfloor t/h\rfloor h})}^2 \, \D t\Bigr).
    \end{align}
    Then, applying ~\eqref{eq:holder} and the AM{--}GM inequality,
    \begin{align*}
        \Bigl\{\E^{Q_T}\bigl[\bigl(\frac{\D P_T}{\D Q_T}\bigr)^q\bigr]\Bigr\}^2
        &\le \E^{Q_T} \exp\Bigl(q^2 L^2 \int_0^T \norm{z_t - z_{\lfloor t/h\rfloor h}}^{2s} \, \D t\Bigr) \\
        &= \E^{Q_T} \exp\Bigl(q^2 L^2 \sum_{k=0}^{N-1} \int_{kh}^{(k+1)h} \norm{z_t - z_{kh}}^{2s} \, \D t\Bigr) \\
        &= \E^{Q_T} \prod_{k=0}^{N-1} \exp\Bigl(q^2 L^2 \int_{kh}^{(k+1)h} \norm{z_t - z_{kh}}^{2s} \, \D t\Bigr) \\
        &\le \frac{1}{N} \sum_{k=0}^{N-1} \E^{Q_T} \exp\Bigl(q^2 NL^2 \int_{kh}^{(k+1)h} \norm{z_t - z_{kh}}^{2s} \, \D t\Bigr)\,.
    \end{align*}

    Apply Lemma~\ref{lem:stoc_calc_result} (with $\lambda = q^2 L^2 T$) to obtain
    \begin{align*}
        \E\exp\Bigl(q^2 N L^2 \int_{kh}^{(k+1)h} \norm{z_t - z_{kh}}^{2s} \, \D t\Bigr)
        &= \E \exp\Bigl(O\bigl( h^{2s} q^2 L^{2\,(1+s)} T \,(1+\norm{z_{kh}}^{2s^2}) + d^s h^s q^2 L^2 T \bigr)\Bigr)\,.
    \end{align*}
    From Proposition~\ref{prop:main_sg_bd}, we know that each $\norm{z_{kh}} \le R_\delta$ with probability at least $1-\delta$, where
    \begin{align*}
        R_\delta
        &\le 2\mf m + 490\sqrt{T\eu R_2(\mu_0 \mmid \hat \pi)} + \frac{230 h^{1/2}\mf m \, (L+T^{-1}) \, T^{1/2}}{d^{1/2}} + 160\sqrt{T\ln \frac{1}{\delta}}\,.
    \end{align*}
    Note that there is no need for a union bound, since we are bounding each of these expectations separately. Integrating the tail bound, so long as $h \lesssim \frac{1}{q^{1/s}L^{2/s} T^{1/s}},$
    \begin{align*}
        &\E\exp\Bigl(q^2 N L^2 \int_{kh}^{(k+1)h} \norm{z_t - z_{kh}}^{2s} \, \D t\Bigr) \\
        &\qquad = \exp O_s\Bigl(h^{2s} q^2 L^{2\,(1+s)} T\,\Bigl(\mf m + \sqrt{T \eu R_2(\mu_0 \mmid \hat\pi)} + \frac{h^{1/2} \mf m \, (L+T^{-1})\,T^{1/2}}{d^{1/2}}\Bigr)^{2s^2} + d^s h^s q^2 L^2 T\Bigr)\,.
    \end{align*}
    We now choose $h$ in order to make this $\le \exp((q-1)\, \varepsilon)$. This is accomplished by taking
    \begin{align}
        h = O_s\Bigl( \frac{\varepsilon^{1/s}}{dq^{1/s} L^{2/s} T^{1/s}} \min\Bigl\{1, \frac{d}{\mf m^s}, \frac{d}{{\eu R_2(\mu_0 \mmid \hat\pi)}^{s/2}}, \frac{d^{(2s+2)/(s+2)}}{\mf m^{2s/(s+2)}} \Bigr\}\Bigr)\,.
    \end{align}
    The last term in the minimum can be eliminated; indeed, if $d^{(2s+2)/(s+2)}/\mf m^{2s/(s+2)} \ge 1$, then it is not active in the minimum.
    Otherwise, raising this expression to the power $(s+2)/2 \ge 1$,
    \begin{align*}
        \frac{d^{(2s+2)/(s+2)}}{\mf m^{2s/(s+2)}}
        &\ge \frac{d^{s+1}}{\mf m^s}
        \ge \frac{d}{\mf m^s}\,.
    \end{align*}
    With the above choice of step-size, this bound gives
    \begin{align*}
        \E^{Q_T}\bigl[\bigl(\frac{\D P_T}{\D Q_T}\bigr)^q\bigr]
        &\le \exp((q-1)\,\varepsilon)\,,
    \end{align*}
    and taking logarithms, we obtain $\eu R_q(P_T \mmid Q_T) \le \varepsilon$ as desired.
\end{proof}

\subsubsection{Finishing the proof}

Finally, we use Theorem~\ref{thm:lo_renyi} on the continuous-time convergence of the Langevin diffusion~\eqref{eq:langevin} in R\'enyi divergence under an LO inequality. Together with our discretization bound, it will imply Theorem~\ref{thm:main}.

\begin{lemma}\label{lem:cont_time_convergence}
    Let ${(\pi_t)}_{t\ge 0}$ denote the law of the continuous-time diffusion~\eqref{eq:langevin} initialized at $\mu_0$, and assume that $\pi$ satisfies~\eqref{eq:lo} with order $\alpha$. If
    \begin{align*}
        T \ge 68q \CLOI \, \Bigl( \frac{{\eu R_q(\mu_0 \mmid \pi)}^{2/\alpha - 1} - 1}{2/\alpha - 1} + \ln \frac{1}{\varepsilon}\Bigr)\,,
    \end{align*}
    we obtain $\eu R_q(\pi_T \mmid \pi) \le \varepsilon$.
\end{lemma}
\begin{proof}
    Recall from Theorem~\ref{thm:lo_renyi} that
    \begin{align*}
        \partial_t \eu R_q(\pi_t \mmid \pi)
        &\le - \frac{1}{68q \CLOI} \times \begin{cases} {\eu R_q(\pi_t \mmid \pi)}^{2-2/\alpha}\,, & \text{if}~\eu R_q(\pi_t \mmid \pi) \ge 1\,, \\ \eu R_q(\pi_t \mmid \pi)\,, & \text{if}~\eu R_q(\pi_t \mmid \pi) \le 1\,. \end{cases}
    \end{align*}
    In general, if $R : \R_+ \to \R_+$ satisfies the ODE $R' = -CR^\beta$ for some $\beta \in (0, 1)$, then a calculation shows that
    \begin{align*}
        R(t)
        &= {\{{R(0)}^{1-\beta} - C \, (1-\beta) \, t\}}^{1/(1-\beta)}\,.
    \end{align*}
    Thus, if $\alpha < 2$, we obtain $\eu R_q(\pi_{T_0} \mmid \pi) \le 1$ at time
    \begin{align*}
        T_0 = \frac{68q \CLOI}{2/\alpha - 1} \, \{{\eu R_q(\mu_0 \mmid \pi)}^{2/\alpha - 1} - 1\}\,.
    \end{align*}
    Observe that as $\alpha \to 2$, then $T_0 \to 68qC_{\msf{LOI}(2)} \ln \eu R_q(\mu_0 \mmid \pi)$ which recovers the continuous-time convergence under~\eqref{eq:lsi}.
    Then, at time $T = T_0 + 68q \CLOI \ln(1/\varepsilon)$, we obtain $\eu R_q(\pi_T \mmid \pi) \le \varepsilon$.
\end{proof}

\begin{proof}[Proof of Theorem~\ref{thm:main}]
    Let ${(\mu_t)}_{t\ge 0}$ denote the law of the interpolated process~\eqref{eq:interpolated_lmc} and let ${(\pi_t)}_{t\ge 0}$ denote the law of the continuous-time Langevin diffusion~\eqref{eq:langevin}, both initialized at $\mu_0$.
    By the weak triangle inequality (when $q \geq 2$), we can bound
    \begin{align*}
        \eu R_q(\mu_{Nh} \mmid \pi)
        &\lesssim \eu R_{2q}(\mu_{Nh} \mmid \pi_{Nh}) + \eu R_{2q-1}(\pi_{Nh} \mmid \pi)\,.
    \end{align*}
    For $T \deq Nh$, we can make the second term at most $\varepsilon/2$ if we choose
    \begin{align*}
        T = \widetilde \Theta\bigl(q\CLOI \, {\eu R_{2q-1}(\mu_0 \mmid \pi)}^{2/\alpha - 1}\bigr)
    \end{align*}
    by Lemma~\ref{lem:cont_time_convergence}.
    Then, by Theorem~\ref{thm:main_disc_bd}, we can make the first term at most $\varepsilon/2$ taking
    \begin{align}\label{eq:final_step_size}
    \begin{aligned}
        h
        &= \widetilde \Theta_s\Bigl( \frac{\varepsilon^{1/s}}{dq^{2/s} \CLOI^{1/s} L^{2/s} \, {\eu R_{2q-1}(\mu_0 \mmid \pi)}^{(2/\alpha - 1)/s}} \\
        &\qquad\qquad\qquad\times{} \min\Bigl\{1, \frac{1}{q^{1/s} \varepsilon^{1/s}}, \frac{d}{\mf m^s}, \frac{d}{{\eu R_2(\mu_0 \mmid \hat\pi)}^{s/2}} \Bigr\}\Bigr)\,.
        \end{aligned}
    \end{align}
    Then, the total number of iterations of \ref{eq:lmc} is
    \begin{align*}
        N
        = \frac{T}{h}
        &= \widetilde \Theta_s\Bigl( \frac{dq^{1+2/s} \CLOI^{1+1/s} L^{2/s} \, {\eu R_{2q-1}(\mu_0 \mmid \pi)}^{(2/\alpha - 1) \, (1+1/s)}}{\varepsilon^{1/s}} \\
        &\qquad\qquad\qquad\qquad\qquad\times{} \max\Bigl\{1, q^{1/s} \varepsilon^{1/s}, \frac{\mf m^s}{d}, \frac{{\eu R_2(\mu_0 \mmid \hat\pi)}^{s/2}}{d} \Bigr\}\Bigr)\,.
    \end{align*}
    This completes the proof.
\end{proof}

\subsection{Proof of Theorems~\ref{thm:mlsi_cont_time} and \ref{thm:main_mlsi}} \label{scn:mlsi}

We first prove the continuous-time convergence for the Langevin diffusion~\eqref{eq:langevin} under~\eqref{eq:mlsi} and~\eqref{eq:tail}. \medskip

\begin{proof}[Proof of Theorem~\ref{thm:mlsi_cont_time}]
    From~\cite[Lemma 6]{vempalawibisono2019ula}, we have
    \begin{align*}
        \partial_t \eu R_q(\pi_t \mmid \pi)
        &= - \frac{4}{q} \, \frac{\E_\pi[\norm{\nabla(\rho_t^{q/2})}^2]}{\E_\pi(\rho_t^q)}\,,
    \end{align*}
    where $\rho_t \deq \frac{\D \pi_t}{\D \pi}$. Applying~\eqref{eq:mlsi} to $f^2 =  \rho_t^q/\E_\pi(\rho_t^q)$,
    \begin{align}\label{eq:mlsi-deriv}
\begin{aligned}
        \frac{4}{q} \, \frac{\E_\pi[\norm{\nabla(\rho_t^{q/2})}^2]}{\E_\pi(\rho_t^q)}
        &\ge \frac{4}{q} \, \Bigl(\frac{\mathsf{ent}_\pi(\rho_t^q)}{2C_{\mathsf{MLSI}} \E_\pi(\rho_t^q) \, {\widetilde{\mathfrak m}_p((1 + \rho_t^q/\E_\pi(\rho_t^q))\pi)}^{\delta(p)}}\Bigr)^{1/(1-\delta(p))} \\
        &\ge \frac{1}{q C_{\mathsf{MLSI}}^2 \, {\widetilde{\mathfrak m}_p((1 + \rho_t^q)\pi)}^{\delta(p)/(1-\delta(p))}} \, \bigl(\frac{\mathsf{ent}_\pi(\rho_t^q)}{\E_\pi(\rho_t^q)}\bigr)^{1/(1-\delta(p))} \\
        &\ge \frac{1}{q C_{\mathsf{MLSI}}^2 \, {\widetilde{\mathfrak m}_p((1 + \rho_t^q)\pi)}^{\delta(p)/(1-\delta(p))}} \, {\eu R_q(\pi_t \mmid \pi)}^{1/(1-\delta(p))} \\
        &\ge \frac{\varepsilon^{\delta(p)/(1-\delta(p))}}{q C_{\mathsf{MLSI}}^2 \, {\widetilde{\mathfrak m}_p((1 + \rho_t^q)\pi)}^{\delta(p)/(1-\delta(p))}} \, \eu R_q(\pi_t \mmid \pi)
\end{aligned}
\end{align}
We explain each of these steps in depth, beginning with the first equality.
Let $f^2 =  \rho_t^q/\E_\pi(\rho_t^q)$. Then, substituting $f^2$ into \eqref{eq:mlsi}, we find immediately that
\begin{align*}
    \mathsf{ent}_\pi(f^2)
    &\le 2C_{\mathsf{MLSI}}\; \inf_{p\ge 2} \bigl\{{\E_\pi[\norm{\nabla f}^2]}^{1-\delta(p)} \, {\widetilde{\mathfrak m}_p\bigl((1+f^2)\pi\bigr)}^{\delta(p)}\bigr\}\,.
\end{align*}
The quantity $\E_\pi[\norm{\nabla f}^2] = \frac{\E_\pi[\norm{\nabla(\rho_t^{q/2})}^2]}{\E_\pi(\rho_t^q)}$, which can be seen by substitution.
Thus, rearranging the terms, we find
\begin{align*}
    \frac{\E_\pi[\norm{\nabla(\rho_t^{q/2})}^2]}{\E_\pi(\rho_t^q)} \geq \Bigl(\frac{\mathsf{ent}_\pi(f^2)}{2C_{\mathsf{MLSI}} \, {\widetilde{\mathfrak m}_p((1 + \rho_t^q/\E_\pi(\rho_t^q))\pi)}^{\delta(p)}}\Bigr)^{1/(1-\delta(p))}\,.
\end{align*}
It remains only to substitute our definition of $f^2$, and use $\mathsf{ent}_\pi(g/C) = \mathsf{ent}_\pi(g)/C$ for all functions $g$ and constants $C$, and we arrive at the first equality.
Then, the second step in \eqref{eq:mlsi-deriv} is straightforward. The third step uses that $\frac{\mathsf{ent}_\pi(\rho_t^q)}{\E_\pi(\rho_t^q)} \geq \eu R_q(\pi_t \mmid \pi)$, since $\mathsf{ent}_\pi(f) = \E_\pi[f \ln(f/\E_\pi(f))]$ implies
\begin{align}\label{eq:ent-renyi-equiv}
\begin{aligned}
    \frac{\mathsf{ent}_\pi(\rho_t^q)}{\E_\pi(\rho_t^q)} &= \E_\pi[(\rho_t^q /\E_\pi(\rho_t^q)) \ln(\rho_t^q/\E_\pi(\rho_t^q))] \\
    &= q\,\partial_q \ln \E_\pi[\rho_t^q] - \ln \E_\pi[\rho_t^q] \\
    &= q\,\partial_q \bigl((q-1) \eu R_q(\pi_t \mmid \pi) \bigr) - (q-1)\, \eu R_q(\pi_t \mmid \pi) \\
    &= q\, \eu R_q(\pi_t \mmid \pi) + q\,(q-1)\, \partial_q\eu R_q(\pi_t \mmid \pi) - (q-1) \, \eu R_q(\pi_t \mmid \pi) \\
    &= \eu R_q(\pi_t \mmid \pi)+ q\,(q-1)\, \partial_q\eu R_q(\pi_t \mmid \pi) \geq \eu R_q(\pi_t \mmid \pi)\,.
\end{aligned}
\end{align}
This closely follows the derivation in~\cite[Lemma 5]{vempalawibisono2019ula}. The last step of \eqref{eq:ent-renyi-equiv} follows as $\partial_q \eu R_q(\pi_t \mmid \pi) \geq 0$, since $\eu R_q$ is increasing in $q$. Thus, the third equality of \eqref{eq:mlsi-deriv} holds. The last step of \eqref{eq:mlsi-deriv} assumes that $\eu R_q(\pi_t \mmid \pi) \geq \varepsilon$, since otherwise the claim immediately holds.
Next, we bound the moments. It is a standard exercise~\cite[see][Exercise 2.7.3]{vershynin2018highdimprob} to show that~\eqref{eq:tail} implies ${\widetilde{\mf m}_p(\pi)}^{1/p} \lesssim \mf m + C_{\msf{tail}}\, p^{1/\alpha_1}$. Also, by a slight modification of the change of measure principle (Lemma~\ref{lem:change_of_measure}), we can show that ${\widetilde{\mf m}_p(\rho_t^q \pi)}^{1/p} \lesssim \mf m + C_{\msf{tail}} \, {\eu R_2(\rho_t^q \pi \mmid \pi)}^{1/\alpha_1} + C_{\msf{tail}}\, p^{1/\alpha_1}$, and that $\eu R_2(\rho_t^q \pi \mmid \pi) \lesssim q \eu R_{2q}(\pi_t \mmid \pi) \le q \eu R_{2q}(\pi_0 \mmid \pi)$. Therefore,
    \begin{align*}
        &{\widetilde{\mf m}_p\bigl((1 + \rho_t^q)\pi\bigr)}^{\delta(p)/(1-\delta(p))}
        \le {\widetilde{\mf m}_p(\pi)}^{\delta(p)/(1-\delta(p))} + {\widetilde{\mf m}_p(\rho_t^q \pi)}^{\delta(p)/(1-\delta(p))} \\
        &\qquad \lesssim {\{\mf m + qC_{\msf{tail}} \, {\eu R_{2q}(\pi_0 \mmid \pi)}^{1/\alpha_1} + C_{\msf{tail}}\, p^{1/\alpha_1}\}}^{(2-\alpha_0) \, (1 + \alpha_0/(p-\alpha_0))}\,.
    \end{align*}
    Using the assumption that $\mf m, C_{\msf{tail}}, \eu R_{2q}(\pi_0 \mmid \pi) \le d^{O(1)}$, if we choose $p\gtrsim \log d$, then
    \begin{align*}
        {\widetilde{\mf m}_p\bigl((1 + \rho_t^q)\pi\bigr)}^{\delta(p)/(1-\delta(p))}
        &\lesssim {\{\mf m + qC_{\msf{tail}} \, {\eu R_{2q}(\pi_0 \mmid \pi)}^{1/\alpha_1} + C_{\msf{tail}}\, p^{1/\alpha_1}\}}^{2-\alpha_0}\,.
    \end{align*}
    Together, it implies that $\eu R_q(\pi_T \mmid \pi) \le \varepsilon$ whenever
    \begin{align*}
        T \geq C\,\Bigl(\frac{qC_{\msf{MLSI}}^2}{\varepsilon^{2\delta(p)}}\, {\{\mf m + qC_{\msf{tail}} \, {\eu R_{2q}(\pi_0 \mmid \pi)}^{1/\alpha_1} + C_{\msf{tail}}\, p^{1/\alpha_1}\}}^{2-\alpha_0} \ln \frac{{\eu R_{q}(\pi_0 \mmid \pi)}}{\varepsilon} \Bigr)\,,
    \end{align*}
    for some absolute constant $C > 0$.
    Next, choosing $p \asymp \ln(d/\varepsilon)$, we obtain $\varepsilon^{2\delta(p)} \gtrsim 1$, so that
    \begin{align*}
    T \ge C\,\Bigl(qC_{\msf{MLSI}}^2\, {\{\mf m + qC_{\msf{tail}} \, {\eu R_{2q}(\pi_0 \mmid \pi)}^{1/\alpha_1} + C_{\msf{tail}} \, {\ln(d/\varepsilon)}^{1/\alpha_1}\}}^{2-\alpha_0} \ln \frac{{\eu R_{q}(\pi_0 \mmid \pi)}}{\varepsilon} \Bigr)\,,
    \end{align*}
    where $C > 0$ is again an absolute constant. This
    completes the proof.
\end{proof}

With the continuous-time result in hand, it is now straightforward to combine it with the discretization result (Theorem~\ref{thm:main_disc_bd}) from the previous section.
\medskip

\begin{proof}[{Proof of Theorem~\ref{thm:main_mlsi}}]
    Let ${(\mu_t)}_{t\ge 0}$ denote the law of the interpolated process~\eqref{eq:interpolated_lmc} and let ${(\pi_t)}_{t\ge 0}$ denote the law of the continuous-time Langevin diffusion~\eqref{eq:langevin}, both initialized at $\mu_0$.
    By the weak triangle inequality (when $q\geq 2$), we can bound
    \begin{align*}
        \eu R_q(\mu_{Nh} \mmid \pi)
        &\lesssim \eu R_{2q}(\mu_{Nh} \mmid \pi_{Nh}) + \eu R_{2q-1}(\pi_{Nh} \mmid \pi)\,.
    \end{align*}
    For $T \deq Nh$, we can make the second term at most $\varepsilon/2$ if we choose
    \begin{align*}
        T
        &= \widetilde \Theta(qC_{\msf{MLSI}}^2\, {\{\mf m + qC_{\msf{tail}} \, {\eu R_{2q}(\mu_0 \mmid \pi)}^{1/\alpha_1}\}}^{2-\alpha_0})
    \end{align*}
    by Theorem~\ref{thm:mlsi_cont_time}.
    Then, by Theorem~\ref{thm:main_disc_bd}, we can make the first term at most $\varepsilon/2$ taking
    \begin{align}\label{eq:final_step_size_mlsi}
        \begin{aligned}
        h
        &= \widetilde \Theta_s\Bigl( \frac{\varepsilon^{1/s}}{dq^{(4-\alpha_0)/s} C_{\msf{MLSI}}^{2/s} C_{\msf{tail}}^{(2-\alpha_0)/s} L^{2/s} \, {\eu R_{2q}(\mu_0 \mmid \pi)}^{(2-\alpha_0)/(\alpha_1 s)}} \\
        &\qquad\qquad\qquad{} \times \min\Bigl\{1, \frac{1}{q^{1/s} \varepsilon^{1/s}}, \frac{d}{\mf m^s}, \frac{d}{{\eu R_2(\mu_0 \mmid \hat\pi)}^{s/2}}, \bigl( \frac{\eu R_{2q}(\mu_0 \mmid \pi)^{1/\alpha_1}}{\mf m} \bigr)^{(2-\alpha_0)/s} \Bigr\}\Bigr)\,.
        \end{aligned}
    \end{align}
    Then, the total number of iterations of \ref{eq:lmc} is
    \begin{align*}
        N
        &= \frac{T}{h} \\
        &= \widetilde \Theta_s\Bigl( \frac{dq^{(1 + (3-\alpha_0) \, (1+s))/s} C_{\msf{MLSI}}^{2 \, (1+1/s)} C_{\msf{tail}}^{(2-\alpha_0) \, (1+1/s)} L^{2/s} \, {
        \eu R_{2q}(\mu_0 \mmid \pi)}^{(2-\alpha_0) \, (1+1/s)/\alpha_1}}{\varepsilon^{1/s}} \\
        &\qquad\qquad\qquad{} \times \max\Bigl\{1, q^{1/s} \varepsilon^{1/s}, \frac{\mf m^s}{d}, \frac{{\eu R_2(\mu_0 \mmid \hat\pi)}^{s/2}}{d}, \bigl( \frac{\mf m}{\eu R_{2q}(\mu_0 \mmid \pi)^{1/\alpha_1}} \bigr)^{(2-\alpha_0)/s} \Bigr\}\Bigr)\,.
    \end{align*}
    This completes the proof.
\end{proof}

\section{Conclusion}\label{scn:conclusion}

In this work, we have given a suite of sampling guarantees for the \ref{eq:lmc} algorithm which assume only that a functional inequality and a smoothness condition hold. In particular, no such guarantees were previously known beyond the LSI case considered in~\cite{vempalawibisono2019ula}. Consequently, we have resolved the open questions of estimating the R\'enyi bias of LMC (Corollary~\ref{cor:renyi_bias}) and establishing quantitative convergence guarantees for LMC under a Poincar\'e inequality. Our results and techniques are also of interest because they work with a stronger metric (namely, R\'enyi divergence) than what is usually considered in the sampling literature.

To conclude, we list a few directions for future research.
\begin{itemize}
    \item Towards the goal of understanding non-log-concave sampling, it is important to establish sampling guarantees for other algorithms, such as underdamped Langevin Monte Carlo, under suitable functional inequalities. Similarly, it is not clear how sharp our bounds are, and it is worth investigating whether our techniques can be improved.
    \item As discussed in the introduction, obtaining guarantees in R\'enyi divergence is useful for applications to differential privacy, as well as for obtaining warm starts for high-accuracy algorithms. Hence, we ask whether R\'enyi convergence guarantees can be proved for more sophisticated algorithms, such as randomized midpoint discretizations~\cite{shenlee2019randomizedmidpoint, hebalasubramanianerdogdu2020rm}.
\end{itemize}

\newpage
\appendix

\section{Initialization}\label{scn:initialization}

In this section, we give bounds on the R\'enyi divergence at initialization. We begin with the convex case.

\begin{lemma}\label{lem:initialization_cvx}
    Suppose that $V$ is convex with $V(0) = 0$ and $\nabla V(0) = 0$, and assume that $\nabla V$ is $L$-Lipschitz. Let $\mf m \deq \int \norm \cdot \, \D \pi$. Then, for $\mu_0 = \normal(0, L^{-1} I_d)$,
    \begin{align*}
        \eu R_\infty(\mu_0 \mmid \pi)
        &\le 2 + \frac{d}{2} \ln(2\mf m^2 L)\,.
    \end{align*}
\end{lemma}
\begin{proof}
    We can write
    \begin{align}\label{eq:renyi_initialization_decomposition}
        \sup \frac{\mu_0}{\pi}
        &= \sup_{x\in\R^d} \exp\bigl\{V(x) - \frac{L}{2} \, \norm x^2\bigr\} \, \frac{\int\exp(-V)}{\int\exp(-V-\delta\, \norm \cdot^2)} \, \frac{\int \exp(-V-\delta \,\norm \cdot^2)}{{(2\uppi/L)}^{d/2}}
    \end{align}
    for some $\delta > 0$ to be chosen later.
    We bound the three ratios in turn. First,
    \begin{align*}
        \exp\bigl\{V(x) - \frac{L}{2} \, \norm x^2\bigr\}
        \le 1
    \end{align*}
    using $V(x) \le L \, \norm x^2/2$. Next,
    \begin{align*}
        \frac{\int\exp(-V - \delta \, \norm \cdot^2)}{\int\exp(-V)}
        &= \int \exp(-\delta \, \norm \cdot^2) \, \D \pi
        \ge \exp(-4\delta \mf m^2) \, \pi\{\norm \cdot \le 2\mf m\}
        \ge \frac{1}{2}\exp(-4\delta \mf m^2)
    \end{align*}
    by Markov's inequality. Finally, since $V\ge 0$,
    \begin{align*}
        \frac{\int \exp(-V-\delta \,\norm \cdot^2)}{{(2\uppi/L)}^{d/2}}
        &\le \frac{\int \exp(-\delta \,\norm \cdot^2)}{{(2\uppi/L)}^{d/2}}
        = \bigl(\frac{L}{2\delta}\bigr)^{d/2}\,.
    \end{align*}
    Taking $\delta = 1/(4\mf m^2)$, we obtain
    \begin{align*}
        \eu R_\infty(\mu_0 \mmid \pi)
        &= \ln \sup \frac{\mu_0}{\pi}
        \le 2 + \frac{d}{2} \ln(2\mf m^2 L)\,,
    \end{align*}
    which is $O(d)$, up to a logarithmic factor.
\end{proof}

We next extend this result to the general case.

\begin{lemma}\label{lem:initialization_general}
    Suppose that $\nabla V(0) = 0$ and that $\nabla V$ satisfies~\eqref{eq:holder} with constant $L > 0$. Let $\mf m \deq \int \norm \cdot \, \D \pi$. Then, for $\mu_0 = \normal(0, {(2L)}^{-1} I_d)$,
    \begin{align*}
        \eu R_\infty(\mu_0 \mmid \pi)
        &\le 2 + L + V(0) - \min V + \frac{d}{2} \ln(4\mf m^2 L)\,.
    \end{align*}
\end{lemma}
\begin{proof}
    We consider the same decomposition as in~\eqref{eq:renyi_initialization_decomposition}.
    First, for some $\lambda \in [0, 1]$, we have
    \begin{align*}
        \abs{V(x) - V(0)}
        &= \abs{\langle \nabla V(\lambda x), x\rangle}
        \le \norm{\nabla V(\lambda x) - \nabla V(0)} \, \norm x
        \le L \, \norm x^{1+s}\,.
    \end{align*}
    Therefore,
    \begin{align*}
        \exp\{V(x) - L \, \norm x^2\}
        &\le \exp\{V(x) - V(0) + V(0) - L \,\norm x^2\} \\
        &\le \exp\{V(0) + L \, \norm x^{1+s} - L \, \norm x^2\}
        \le \exp\{V(0) + L\}
    \end{align*}
    using $t^{1+s} \le 1 + t^2$ for all $t\ge 0$.
    Next,
    \begin{align*}
        \frac{\int\exp(-V - \delta \, \norm \cdot^2)}{\int\exp(-V)}
        &\ge \frac{1}{2}\exp(-4\delta \mf m^2)
    \end{align*}
    as before. Lastly,
    \begin{align*}
        \frac{\int \exp(-V-\delta \,\norm \cdot^2)}{{(\uppi/L)}^{d/2}}
        &\le \frac{\exp(-\min V) \int \exp(-\delta \,\norm \cdot^2)}{{(\uppi/L)}^{d/2}}
        = \exp(-\min V) \, \bigl(\frac{L}{\delta}\bigr)^{d/2}\,.
    \end{align*}
    This yields
    \begin{align*}
        \eu R_\infty(\mu_0 \mmid \pi)
        &= \ln \sup \frac{\mu_0}{\pi}
        \le 2 + L + V(0) - \min V + \frac{d}{2} \ln(4\mf m^2 L)\,,
    \end{align*}
    with the choice $\delta = 1/(4\mf m^2)$.
\end{proof}

In order to obtain an initialization with $\eu R_\infty(\mu_0 \mmid \pi) = \widetilde O(d)$, the lemma requires finding a stationary point $x\in\R^d$ such that the optimality gap $V(x) - \min V$ is not too large, i.e., of order $O(d)$. Since $\nabla V$ satisfies~\eqref{eq:holder}, it suffices to find a stationary point which lies in a ball of radius $O(d^{1/(1+s)})$ centered at the minimizer of $V$. Based on this result, it seems reasonable to assume that the initialization typically satisfies $\eu R_\infty(\mu_0 \mmid \pi) = \widetilde O(d)$.

Actually, in the setting of Theorem~\ref{thm:main}, we also need a bound on the R\'enyi divergence $\eu R_2(\mu_0 \mmid\hat\pi)$, where $\hat\pi$ is a slight modification of $\pi$ (see Section~\ref{scn:proof_main}). The following lemma is proven just as in Lemma~\ref{lem:initialization_general}, so the proof is omitted.

\begin{lemma}\label{lem:initialization_pseudo}
    Suppose that $\nabla V(0) = 0$ and that $\nabla V$ satisfies~\eqref{eq:holder} with constant $L > 0$. For some $\gamma > 0$, let $\hat V(x) \deq V(x) + \frac{\gamma}{2} \, {(\norm x - R)}_+^2$, and let $\hat\pi \propto \exp(-\hat V)$. Also, let $\hat{\mf m} \deq \int \norm \cdot \, \D \hat\pi$. Then, for $\mu_0 = \normal(0, {(2L + \gamma)}^{-1} I_d)$,
    \begin{align*}
        \eu R_\infty(\mu_0 \mmid \hat\pi)
        &\le 2 + L + \frac{\gamma}{2} + V(0) - \min V + \frac{d}{2} \ln(4\hat{\mf m}^2 L)\,.
    \end{align*}
\end{lemma}

From the tail bound in Lemma~\ref{lem:modified_potential}, we can deduce an upper bound for $\hat{\mf m}$ as follows 
\begin{equation*}
\begin{aligned}
    \hat{\mf m} 
    &= \int_0^\infty \hat \pi ( \|\cdot\| \geq t ) \, \D t \\ 
    &= \int_0^R \hat \pi ( \|\cdot\| \geq t ) \, \D t 
        + \int_0^\infty \hat \pi ( \|\cdot\| \geq R + \eta ) \, \D \eta \\ 
    &\leq 
        R + \int_0^\infty 2 \exp\Bigl( -\frac{\gamma \eta^2}{2} \Bigr) \, \D \eta
        \\ 
    &\lesssim 
        R + \sqrt{ \frac{1}{\gamma} } \,. 
\end{aligned}
\end{equation*}
In Proposition~\ref{prop:main_sg_bd}, we eventually take $\gamma$ roughly of order $1/d \lesssim \gamma \lesssim 1$, and $R \lesssim \mf m$. Hence, if $L + V(0) - \min V = \widetilde O(d)$ and $\mf m \le d^{O(1)}$, then $\eu R_\infty(\mu_0 \mmid \hat\pi) = \widetilde O(d)$.

\section{Additional technical lemmas}\label{scn:technical_lemmas}

In this section, we collect together technical lemmas which appear in the proofs of Section~\ref{scn:proof_main}. The proofs rely on standard arguments from stochastic calculus. The first lemma extends~\cite[Lemma 23]{chewietal2021mala}.

\begin{lemma}\label{lem:brownian_mgf}
    Let ${(B_t)}_{t\ge 0}$ be a standard Brownian motion in $\R^d$.
    Then, if $\lambda \ge 0$ and $h \le 1/(4\lambda)$,
    \begin{align*}
        \E \exp\bigl(\lambda \sup_{t\in [0,h]}{\norm{B_t}^2}\bigr)
        &\le \exp(6dh\lambda)\,.
    \end{align*}
    In particular, for all $\eta \ge 0$,
    \begin{align*}
        \Pr\bigl\{\sup_{t \in [0,h]}{\norm{B_t}} \ge \eta\bigr\}
        &\le 3\exp\bigl(-\frac{\eta^2}{6dh}\bigr)\,.
    \end{align*}
    Next, for $s\in (0,1)$ and $0 \le \lambda < 1/{(12dh)}^s$,
    \begin{align*}
        \E\exp\bigl(\lambda \sup_{t\in [0,h]}{\norm{B_t}^{2s}}\bigr)
        &\le \exp(144 d^s h^s \lambda)\,.
    \end{align*}
\end{lemma}
\begin{proof}
    The first statement follows from~\cite[Lemma 23]{chewietal2021mala}, and the second follows from the first by taking $\lambda = 1/(6dh)$ and applying Markov's inequality.
    
    We now turn towards the proof of the third statement.
    Using the tail bound
    \begin{align*}
        \Pr\bigl\{\sup_{t \in [0,h]}{\norm{B_t}^{2s}} \ge \eta\bigr\}
        &\le 3\exp\bigl( - \frac{\eta^{1/s}}{6dh}\bigr)
    \end{align*}
    we now bound $\E\exp(\lambda \sup_{t\in [0,h]}{\norm{B_t}^{2s}})$.
    \begin{align*}
        \E \exp\bigl(\lambda \sup_{t\in [0,h]}{\norm{B_t}^{2s}} \bigr)
        &= 1 + \lambda \int_0^\infty \exp(\lambda \eta) \, \Pr\bigl\{\sup_{t \in [0,h]}{\norm{B_t}^{2s}} \ge \eta\bigr\} \, \D \eta \\
        &\le 1 + 3\lambda \int_0^\infty \exp\bigl(\lambda \eta - \frac{\eta^{1/s}}{6dh}\bigr) \, \D \eta\,.
    \end{align*}
    Split the integral into whether or not $\eta \ge {(12dh\lambda)}^{s/(1-s)}$. For the first part,
    \begin{align*}
        \lambda \int_0^{{(12dh\lambda)}^{s/(1-s)}} \exp(\lambda \eta) \, \D \eta
        &\le {(12dh)}^{s/(1-s)} \lambda^{1/(1-s)} \exp\{{(12dh)}^{s/(1-s)} \lambda^{1/(1-s)}\} \\
        &\le 3\, {(12dh)}^{s/(1-s)} \lambda^{1/(1-s)}
    \end{align*}
    provided that $\lambda \le 1/{(12dh)}^s$.
    For the second part, using the change of variables $\tau = \eta^{1/s}/(12dh)$,
    \begin{align*}
        \lambda \int_{{(12dh\lambda)}^{s/(1-s)}}^\infty \exp\bigl(\lambda \eta - \frac{\eta^{1/s}}{6dh}\bigr) \, \D \eta
        &\le \lambda \int_{{(12dh\lambda)}^{s/(1-s)}}^\infty \exp\bigl(- \frac{\eta^{1/s}}{12dh}\bigr) \, \D \eta \\
        &\leq 
            {(12dh)}^s s\lambda \int_0^\infty \frac{\exp(-\tau)}{\tau^{1-s}} \, \D \tau \\
        &= 
            {(12dh)}^s s\lambda \Gamma(s)
        = 
            {(12dh)}^s \lambda \Gamma(1+s)
        \leq 
            {(12dh)}^s \lambda \,, 
    \end{align*}
    where we used Gautschi's inequality to obtain $\Gamma(1+s) \leq 1$. 
    We have therefore proven
    \begin{align*}
        \E \exp\bigl(\lambda \sup_{t\in [0,h]}{\norm{B_t}^{2s}} \bigr)
        &\le 1 + 9 \, {(12dh)}^{s/(1-s)} \lambda^{1/(1-s)} 
            + 3\, {(12dh)}^s \lambda
        \leq 
            1 + 144 d^s h^s \lambda \,,
    \end{align*}
    which implies the result.
\end{proof}

The following lemma extends~\cite[Lemma 24]{chewietal2021mala}.

\begin{lemma}\label{lem:stoc_calc_result}
    Let ${(z_t)}_{t\ge 0}$ denote the continuous-time Langevin diffusion~\eqref{eq:langevin} started at $z_0$, and assume that the gradient $\nabla V$ of the potential satisfies $\nabla V(0) = 0$ and~\eqref{eq:holder}.
    Also, assume that $h \le 1/(6L)$ and $\lambda \le 1/(96d^s h^s)$. Then,
    \begin{align*}
        \E\exp\bigl(\lambda\sup_{t\in [0,h]}{\norm{z_t - z_0}^{2s}}\bigr)
        \le \exp\{8h^{2s} L^{2s} \, (1+\norm{z_0}^{2s^2}) \, \lambda 
         + 1152 d^s h^s \lambda\}\,.
    \end{align*}
\end{lemma}
\begin{proof}
    Let $f(t) \deq \sup_{r\in [0,t]}{\norm{z_r - z_0}^2}$.
    Then, for $0 \le t \le h$, since $\norm{\nabla V(x)} \le L \, \norm x^s$,
    \begin{align*}
        \norm{z_t - z_0}^2
        &= \Bigl\lVert -\int_0^t \nabla V(z_r) \, \D r + \sqrt 2 \, B_t\Bigr\rVert^2
        \le 2t\int_0^t \norm{\nabla V(z_r)}^2 \, \D r + 4\, \norm{B_t}^2 \\
        &\le 4t\int_0^t \norm{\nabla V(z_r) - \nabla V(z_0)}^2 \, \D r + 4t^2 \, \norm{\nabla V(z_0)}^2 + 4\, \norm{B_t}^2 \\
        &\le 4tL^2 \int_0^t \norm{z_r - z_0}^{2s} \, \D r + 4t^2 L^2 \, \norm{z_0}^{2s} + 4\, \norm{B_t}^2 \\
        &\le 4tL^2 \int_0^t \norm{z_r - z_0}^2 \, \D r + 4t^2 L^2 \, (1+\norm{z_0}^{2s}) + 4\, \norm{B_t}^2\,,
    \end{align*}
    which yields
    \begin{align*}
        f(t)
        &\le 4t^2 L^2 \, (1+\norm{z_0}^{2s}) + 4 \sup_{r\in [0,t]} \norm{B_r}^2 + 4tL^2 \int_0^t f(r) \, \D r\,.
    \end{align*}
    Gr\"onwall's inequality yields
    \begin{align*}
        f(h)
        &\le \bigl(4h^2 L^2 \, (1+\norm{z_0}^{2s}) + 4 \sup_{r\in [0,h]} \norm{B_r}^2\bigr) \exp(2h^2 L^2) \\
        &\le 8h^2 L^2 \, (1+\norm{z_0}^{2s})+ 8 \sup_{r\in [0,h]} \norm{B_r}^2
    \end{align*}
    using $h \le 1/(6 L)$. It also yields
    \begin{align*}
        \sup_{t \in [0,h]}{\norm{z_t - z_0}^{2s}}
        &\le 8h^{2s} L^{2s} \, (1+\norm{z_0}^{2s^2}) + 8 \sup_{r\in [0,h]}{\norm{B_r}^{2s}}\,.
    \end{align*}
    The result now follows from Lemma~\ref{lem:brownian_mgf}.
\end{proof}

\section{Acknowledgements}

We would like to thank Krishnakumar Balasubramanian, Alain Durmus, Holden Lee, Chen Lu, Kevin Tian, and Andre Wibisono for many helpful conversations.

\section{Funding}

SC was supported by the Department of Defense (DoD) through the National Defense Science \& Engineering Graduate Fellowship (NDSEG) Program. 
MAE was supported by NSERC Grant [2019-06167], Connaught New Researcher Award, CIFAR AI Chairs program, and CIFAR AI Catalyst grant.
ML was supported by the Ontario Graduate Scholarship (OGS) and the Vector Institute. 
This work was done while several of the authors were visiting the Simons Institute for the Theory of Computing.

\printbibliography

\end{document}